\documentclass[11pt,twoside,a4paper]{article}
\usepackage{psfrag,color}

\usepackage{amsfonts,amsmath,latexsym,graphicx}
\usepackage[latin1]{inputenc}
\usepackage{a4,longtable,float,hhline,epsfig,bbm}
\usepackage{amsthm,amssymb}
\usepackage[normal]{subfigure}
\usepackage{caption2}
\usepackage{fancyheadings}

\usepackage[algonl,algo2e,algoruled]{algorithm2e}
\usepackage{pseudocode}

\usepackage{soul}
\usepackage[mathscr]{eucal}
\usepackage{textcomp}

\renewcommand{\(}{$\,}
\renewcommand{\)}{\,$}
\newcommand{\cc}[1]{\mathscr{#1}}

\def\eqdef{\stackrel{\operatorname{def}}{=}}

\renewcommand{\hat}[1]{\widehat{#1}}
\renewcommand{\tilde}[1]{\widetilde{#1}}
\def\T{\top}

\def\bm{\varrho}
\def\T{\top}

\def\Var{\operatorname{Var}}
\def\Cov{\operatorname{Cov}}

\def\omegac{\omega'}

\def\omegad{\omega^{\circ}}
\def\entrl{\mathfrak{e}}

\def\kapla{\mathfrak{n}_{0}}

\def\lambdab{\lambda^{*}}

\sloppypar
\newtheorem{lemma}{Lemma}
\newtheorem{theorem}{Theorem}

\def\E{I\!\!E}
\def\P{I\!\!P}

\newcommand{\be}{\begin{eqnarray}}
\newcommand{\ee}[1]{\label{eq:#1}\end{eqnarray}}
\newcommand{\nn}{\nonumber \\}
\newcommand{\ese}{\end{eqnarray*}}
\newcommand{\bse}{\begin{eqnarray*}}
\newcommand{\rf}[1]{~(\ref{eq:#1})}

\newcommand{\R}{\mathbb{R}}

\newcommand{\cB}{\mathcal{B}}

\newcommand{\cE}{\mathcal{E}}

\newcommand{\cI}{\mathcal{I}}

\newcommand{\cN}{\mathcal{N}}
\newcommand{\cO}{\mathcal{O}}

\newcommand{\cT}{\mathcal{T}}
\newcommand{\cU}{\mathcal{U}}

\newcommand{\cX}{\mathcal{X}}

\newcommand{\cZ}{\mathcal{Z}}
\DeclareMathOperator*{\argmin}{\mathrm{arg\,min}}
 
\newcommand{\diag}{\mathrm{diag}}

\newcommand{\Tr}{\mathrm{Tr}}

\numberwithin{figure}{section}

\makeindex

\begin{document}
%
\author{ \(  \text{\sc Elmar Diederichs}^1 \), \(  \text{\sc Anatoli Juditski}^3 \), \\ \( \text{\sc Vladimir Spokoiny}^{2} \), \( \text{\sc Christof Schütte}^1 \) \\
\\ \\ \\
\(  ^1\text{Institute for Mathematics and Informatics, Free University Berlin} \) \\ \( \text{Arnimallee 6, 14195 Berlin, Germany} \) \\ \\
\(  ^2\text{Weierstrass Institute and Humboldt University} \) \\ \( \text{Mohrenstr. 39, 10117 Berlin, Germany} \) \\ \\
\(  ^3\text{LJK, Universit\'e J. Fourier, } \)\\ \( \text{BP 53 38041 GRENOBLE cedex 9, France} \) \\ \\
}
\title{\sc Sparse NonGaussian Component Analysis\footnote{Supported by DFG research center {\sc Matheon} ''Mathematics for key technologies'' (FZT 86) in Berlin.}}
\date{\today}
\maketitle
\thispagestyle{empty}
\begin{abstract}
{\noindent}Non-gaussian component analysis (NGCA) introduced in
\cite{a3-3} offered a method for high dimensional data analysis allowing
for identifying a low-dimensional non-Gaussian component of the whole
distribution in an iterative and structure adaptive way. An important step
of the NGCA procedure is identification of the non-Gaussian subspace using
Principle Component Analysis (PCA) method. This article proposes a new
approach to NGCA called \emph{sparse NGCA} which replaces the PCA-based
procedure with a new the  algorithm we refer to as \emph{convex
projection}.\\[5ex]

{\noindent}{\em keywords:} reduction of dimensionality, model reduction,
sparsity, variable selection, principle component analysis, structural
adaptation, convex projection
\\[3ex]
{\noindent}{\em Mathematical Subject Classification:} 62G05, 60G10, 60G35, 62M10, 93E10
\end{abstract}

\newpage
\section{Introduction}

Numerous mathematical applications in econometrics or biology are
confronted with high dimensional data. Such data sets present new
challenges in data analysis, since often the data have dimensionality
ranging from hundreds to hundreds of thousands. This means an exponential
increase of the computational burden for many methods. On the other hand
the sparsity of the data in high dimensions entails that data thin out in
the local neighborhood of a given point \( x \). Hence statistical methods
are not reliable in high dimensions if the sample size remains of the same
order. This problem is usually referred to as ''curse of dimensionality''
(cf. \cite{Hastie01}, \cite{Wasserman}). The standard approach to deal
with the high dimensional data is to introduce a \emph{structural
assumption} which allows to reduce the complexity or intrinsic dimension
of the data without significant loss of statistical
information \cite{Roweis2000}, \cite{Mizuta04}.\\

Let a random phenomenon is observed in the high dimensional space \(
\R^{d} \) while the intrinsic dimension of this phenomenon is much
smaller, say \( m \). From a geometrical point of view \( m \) is the
dimension of a linear subspace that approximately contains the structure
of the sample data. Alternatively we can consider this structure as a low
dimensional signal embedded in high dimensional noise. Consequently a
lower dimensional, compact representation that according to some
criterion, captures the interesting information in the original data, is
sought. In this paper we assume that we have a sample of data  lying
approximately in a \(  m\leq d \) dimensional linear (target) subspace \(
\cI\subseteq \R^{d} \) of \( \R^{d} \). In order to reduce the problem
dimension one looks for a mapping from the original
data space onto this subspace.\\

In the statistical literature the Gaussian components of the data
distribution are often considered as entropy maximizing and consequently
as non-informative noise \cite{Cover1991}. It is well known that for
high-dimensional clouds of points most low-dimensional projections are
approximately Gaussian \cite{Friedman1984}. 
The {\em Non-Gaussian Component Analysis} (NGCA), introduced in
\cite{a3-3}, is based on the assumption that the structure of the data is
represented by a low dimensional non-Gaussian component of the observation
distribution, as opposed to a full dimensional Gaussian component,
considered as noise. Thus the objective of NGCA is to ''kill the noise''
rather than to describe the whole multidimensional distribution. Note that
the suggested way of treating the Gaussian distribution as a pure nuisance
in general exclude the use of the classical \emph{Principle Component
Analysis} (PCA) which simply searches for the directions with
of largest variance.\\

In the same way as a number of projection methods of feature extraction (e.g.
Projection Pursuit \cite{Huber85}, Partial Least Square Regression
\cite{Wold82,Wold87}, Conditional Minimum Average Variance Estimation
\cite{a3-5} or Sliced Inverse Regression \cite{Li1991,Cook1998,Cook2001}),
when implementing the NGCA we decompose the problem of dimension reduction
into two tasks: the first one is to extract from data a set of vectors which
are close to the target space \( \cI \); the second is to construct a basis of
the target space  from these vectors. These characteristics can also be found
in the unsupervised, data driven approach of SNGCA, presented in this article.
When compared to available dimension reduction methods (e.g. Principal Component
Analysis \cite{Jolliffe01}, Independent Component Analysis \cite{Hyvaer99} or
Singular Spectrum Analysis \cite{Goljandina}) SNGCA does not assume
any a priori knowledge about the density of the original data.\\

The proposed method, as well as NGCA, is an iterative algorithm which is
structure adaptive in the sense that every new step essentially uses the
result of previous iterations. The main difference between NGCA and SNGCA
algorithms lies in the way the information is extracted  from the data.
The algorithm of NGCA heavily relies upon the Euclidean projection and the
PCA of the set of the estimated vectors. In the case when data dimension
is important and the sample size is moderate, computation of the \( l_{2}
\)-projection can amplify the noise. Moreover, when most of the estimated
vectors do not contain information about the space \( \cI \) but are
mainly noise, the results of using the PCA algorithm to extract the basis
of feature space can be very poor. The reason for that is that the PCA
algorithm is known to accumulate the noise. To address this issue the
SNGCA uses convex programming techniques to estimate elements of the
target subspace by ``convex projection'', what allows to bound uniformly
the estimation error. Further, another technique of convex analysis, based
on computation of rounding ellipsoids of the set of estimated vectors, is
used to extract the subspace information. These changes allow the SNGCA
algorithm to treat large families of candidate vectors without increasing
significantly the variance of the estimation
of the target subspace. \\

The paper is organized as follows. First  we describe the
considered set-up in Section~\ref{SFramework} and  discuss the main ideas behind the proposed approach.
The formal description of the algorithm is given in Section~\ref{SAlgorithms}. A
simulation study of the algorithms is presented in Section~\ref{numerics}, where we
compare the performance obtained by SNGCA algorithms and by several other methods of
feature extraction.

\section{Non-Gaussian Component Analysis}\label{SFramework}

\subsection{The setup}
The following setting is due to \cite{a3-3}. Let \( X_{1},...,X_{N} \) be
i.i.d. from a distribution \( \P \) in \( \R^{d} \). We suppose that \( \P
\) possesses a density \( \rho \) with respect to the Lebesgue measure on
\( \R^{d} \), which can be decomposed as follows: \be
    \rho(x)=\phi_{\mu=0,\Sigma}(x)q(Tx).
\ee{modeldensity} Here \( \phi_{\mu,\Sigma} \) stands for the density of
the multivariate normal distribution \( \cN(\mu,\Sigma) \) with parameters
\( \mu\in \R^{d} \) (expectation) and \( \Sigma\in \R^{d\times d} \)
positive definite (covariance matrix). The function \( q:\R^{m}\to \R \)
with \( m\leq d \) has to be nonlinear and smooth. \( T\in \R^{m\times d}
\) is an unknown linear mapping. Naturally, we refer to \( \cI={\rm
range}\; T \) as {\em target} or {\em non-Gaussian subspace}. For the sake
of simplicity let us assume \( \E[X]=0 \) where \( \E[X] \) stands for
the expectation of \( X \). \\

Though the representation  \rf{modeldensity} is not uniquely defined, the
subspace \( \cI \subset \R^{d} \) is well defined as well as the Euclidean
projector \( \Pi^{*} \) on \( \cI \). By analogy with the regression case
\cite{Cook1998,Li1992,Li1991}, we could also call \( \cI \) {\em the
effective dimension reduction space} (EDR-space). We call \( m \) {\em
effective dimension} of the data. In many applications \( m \) is unknown
and has to be recovered from the data. Our task is to recover \( \Pi^{*}
\). The model structure \rf{modeldensity} allows the following
interpretation (cf. \cite{a3-3}) : we can decompose the random vector \( X
\) into two independent components
\begin{eqnarray*}
    X = \Pi^{*} X + (I-\Pi^{*}) X = Z+u,
\end{eqnarray*}
where \( Z \) is a non-Gaussian \( m \)-dimensional signal and \( u \) is \( (d-m) \)-dimensional normal noise. \\

As we have already noticed in the introduction, SNGCA algorithm relies
upon two basic operations: the first is to construct a set of vectors, say
\( \beta_{1},...,\beta_{J} \), which are ''close'' to the target subspace;
the objective of the second is to compute an estimate \( \hat{\Pi} \) of
the Euclidean projector \( \Pi^{*} \) on \( \cI \) using the set \(
\{\beta_{j}\}_{j=1}^J \).

\subsection{Estimation of elements of the target subspace}

\par{\bf Estimation of elements of \( \cI \).}
The implementation of the first step of SNGCA is based on the following result (cf.
Theorem 1 of \cite{a3-3}):
\begin{theorem}\label{magic}
Let \( X \) follow the distribution with the density \( \rho \) which
satisfies \rf{modeldensity} and let \( \E[X]=0 \). Suppose that a function
\( \psi:\,\R^{d}\to \R \) is continuously differentiable. Define \be
    \beta(\psi)
    :=
    \E \bigl[\nabla \psi(X)\bigr]
    =
    \int \nabla \psi(x) \, \rho(x)\, dx ,
\ee{originalbeta}
where \( \nabla \psi \) stands for the gradient of \( \psi \). Then there exists a
vector \( \beta \in \cI \) such that
\begin{eqnarray*}
    \| \beta(\psi) - \beta \|_{2}
    & \leq &
    \big\| \Sigma^{-1} \E[ X \psi(X) ] \big\|_{2}
    \nn
    &=&
    \Big\| \Sigma^{-1} \int x \psi(x)\rho(x)\;dx \Big\|_{2}.
\end{eqnarray*}
In particular,  if \( \E[X \psi(X)]=0\), then \( \beta(\psi)\in \cI \).
\end{theorem}
The bound of Theorem \ref{magic} implies that
\be
    \| (I - \Pi^{*}) \beta(\psi) \|_{2}
    \le
    \Big\| \Sigma^{-1} \int x \psi(x) \rho(x) \; dx \Big\|_{2},
\ee{bounds} where \( I \) is the \( d \)-dimensional identity matrix and
\( \Pi^{*} \) is the orthogonal projector on \( \cI \). \\

Based on this result, \cite{a3-3} suggested the following way of
constructing a set of vectors \( \beta \) which approximate the target
space \( \cI \). Let \( h_{1},...,h_{L} \) be smooth bounded functions on
\( \R^{d} \). Define \( \gamma_{l} = \E[X h_{l}(X)] \) and \( \eta_{l} =
\E[\nabla h_{l}(X)] \). These vectors are not computable because they rely
on the unknown data distribution, but they can be well estimated from the
given data. Next, for any vector \( c \in \R^{L} \), define the vectors \(
\beta(c), \gamma(c) \in \R^{d} \) with
\begin{eqnarray*}
    \beta(c) = \sum_{l=1}^L c_{l} \eta_{l},
    \qquad
    \gamma(c) = \sum_{l=1}^L c_{l} \gamma_{l}
\end{eqnarray*}
Then by Theorem \ref{magic},  \( \beta(c) \in \cI \) conditioned that \(
\gamma(c)=0 \). Indeed, if we set \( \psi(x)=\sum_{l} c_{l} h_{l}(x) \),
then \( \E[X\psi(X)]=0 \), and by \rf{bounds},
\[
    \gamma(c) = \E [\nabla\psi(X)] \in \cI.
\]
The approach of \cite{a3-3} is to compute the vectors of coefficients \( c \in \R^{L} \)
which ensure \( \gamma(c) \approx 0 \) and then to use the corresponding empirical
analogs of \( \beta(c) \) to estimate the target space.
More precisely, given the observations \( X_{1},...,X_{N} \) compute the set of
vectors (empirical counterparts of \( \eta_{l} \) and \( \gamma_{l} \)) according to
\begin{eqnarray}
\label{hathat}
    \hat{\gamma}_{l}
    =
    N^{-1} \sum_{i=1}^{N} X_{i} h_{l}(X_{i}),
    \quad
    \hat{\eta}_{l}
    =
    N^{-1} \sum_{i=1}^{N} \nabla h_{l}(X_{i}). \qquad
\end{eqnarray}
Similarly define for \( c \in \R^{L} \)
\begin{eqnarray*}
    \hat{\beta}(c) = \sum_{l=1}^L c_{l} \hat{\eta}_{l},
    \qquad
    \hat{\gamma}(c) = \sum_{l=1}^L c_{l} \hat{\gamma}_{l} \, .
\end{eqnarray*}
One can expect that for vectors \( c \) with \( \hat{\gamma}(c) = 0 \),
the vectors \( \hat{\beta}(c) \) are ''close'' to \( \cI \). \\

Below we follow a similar way of constructing \( \hat{\beta}(c) \) with an
additional constraint that the considered vectors of coefficients \( c \)
satisfy \( \| c \|_{1} \le 1 \). This constraint allows for both
efficient numerical algorithms and sharp error bounds.\\

The test functions \( h_{l} \) can be generated as follows: let \( \cB_{d}
\) be a unit ball \( \cB_{d}=\{ x \in \R^{d}| : \, \|x\|_{2} \le 1\} \)
and let \( f(x,\omega) \), \( f:\, \cB \times \R^{d} \to \R \) be a
continuously differentiable function. Consider the functions \( h_{l}(x) =
f(x,\omega_{l}) \), for some \( \omega_{l} \in \cB, \;l=1,...,L \). The
choice of family \( f(\cdot,\omega) \) is an important  parameter of the
algorithm design. For instance, in the simulation examples of Section
\ref{numerics} we consider the following families: \be \label{eq:tanh}
    f(x,\omega)
    &=&
    \tanh(\omega^{\T}x) e^{-\alpha\|x\|^{2}_{2}/2},
    \\
    f(x,\omega)
    &=&
    [1+(\omega^{\T}x)^{2}]^{-1} \exp^{\omega^{\T}x - \alpha \|x\|^{2}_{2}/2},
\ee{asymmeg} and \( \omega_{l}, \;l=1,...,L \) are  unit vectors in \(
\R^{d} \). The next result justifies the proposed construction.

\begin{theorem}
\label{empir1}
Suppose that \( f \) is continuously differentiable in \( w \) and for
some fixed constant \( f^{*}_{1} \) and any  \( \omega \in \cB_{d}, \, x \in \R^{d} \)
\begin{eqnarray*}
    \Var \bigl[ X_{j} \, f(X,\omega) \bigr]
    \le
    f^{*}_{1},
    \quad
    \Cov \bigl[ X_{j} \, \nabla_{\omega} f(X,\omega) \bigr]
    \le
    f^{*}_{1} I,
    \nn
    \Var \biggl[ \frac{\partial}{\partial x_{j}} f(X,\omega) \biggr]
    \le
    f^{*}_{1},
    \quad
    \Cov \biggl[ \nabla_{\omega} \frac{\partial}{\partial x_{j}} f(X,\omega) \biggr]
    \le
    f^{*}_{1} I,
\end{eqnarray*}
Consider the (random) set
\be
    \cc{C}
    =
    \bigl\{ c\in \R^{L} : \| c \|_{1} \le 1, \, \hat{\gamma}(c) = 0 \bigr\}.
\ee{setC}
Then for any \( \varepsilon > 0 \) there is a set \( A \subset \Omega \) of probability at
least \( 1-\varepsilon \) such that on \( A \) for all \( c \in \cc{C} \),
\[
    \bigl\|(I-\Pi^{*}) \hat{\beta}(c)\bigr\|_{2}
    \le
    \sqrt{d} \, \delta_{N} \bigl( 1 + \| \Sigma^{-1} \|_{2} \bigr),
\]
where
\begin{eqnarray*}
    \delta_{N}
    =
    N^{-1/2} \inf_{\lambda \le \lambdab_{1} N^{1/2}}
    \bigl\{
        5 \kapla f^{*}_{1} \lambda
        + 2\lambda^{-1} \bigl[ \entrl_{d} + \log(2d/\varepsilon) \bigr]
    \bigr\}
\end{eqnarray*}
and \( \entrl_{d} = 4 d \log 2 \).
\end{theorem}
The proof of the theorem is given in the appendix.
\\

Due to this result, any vector \( c \in \cc{C} \) can be used to produce a
vector \( \hat{\beta}(c) \) which is close to the target subspace \( \cI
\). However, such constructed vectors are only informative if its length
is significant relative to the estimation error.
\medskip

We therefore compute a family of such coefficient vectors \( c \) by
solving the following optimization problems: for a fixed unit vector \(
\xi \in \R^{d} \) called a \emph{probe vector}, find
\be
    \hat{c}
    =
    \argmin_{c \in \R^{L}: \, \| c \|_{1} \le 1} \|\xi - \hat{\eta}(c) \|_{2},
    \mbox{ subject to }
    \hat{\gamma}(c)=0 .
\ee{chat}
where \( \hat{\eta}(c)= \sum_l c_l\hat{\eta}_l\). This is a
convex optimization problem which can be efficiently solved by some
numerical procedures, e.g. by the interior point method. Then we set
\be
    \hat{\beta}
    =
    \hat{\beta}(\hat{c})
    =
    \sum_{l} \hat{c}_{l} \hat{\eta}_{l} \, .
\ee{betaj1}
It can be easily seen that for \( \xi \perp \cI \), the solution \(
\hat{\beta} \) fulfills \( \hat{\beta} \approx 0 \). On the contrary, if
\( \xi \in \cI \), then there is a solution with significantly positive \(
\| \hat{c} \|_{1} \) and \( \| \hat{\beta}(\hat{c}) \|_{2} \). This leads
to the following strategy: In the first step of the algorithm when there
is no information about \( \cI \) available, the probe vectors \(
\xi_{1},...,\xi_{J} \) in \( \R^{d} \) are generated randomly from \(
\cB_{d} \). In the next steps we apply the idea of structural adaptation
by generating the essential part of the vectors \( \xi_{j} \) from the
estimated subspace \( \tilde{\cI} \). For details see Section \ref{SAlgorithms}.
\\

We address now the implementation of the second step of SNGCA -- inferring
the projector \( \Pi^{*} \) on \( \cI \) from estimations \(
\{\hat{\beta}_{j}\}_{j=1}^J \) of elements of \( \cI \).
\\

{\bf Recovering  the target subspace.} Suppose that we are given vectors
\( \hat{\beta}_{1},...,\hat{\beta}_{J} \) which satisfy
\[
    \|\hat{\beta}_{j} - \beta_{j}\|_{2}
    \le
    \bm,
\]
for some \( \beta_{j} \in \cI \), \( j=1,...,J \).
The problem of estimating the subspace
\( \cI \) from \( \hat{\beta}_{j} \) is a special
case of the so called \emph{Reduced Rank Regression} (RRR) problem.
A simple and popular PCA estimate of the projector \( \Pi^{*} \) on \( \cI \)
is given by solving the quadratic optimization problem
\begin{eqnarray*}
    \hat{\Pi}
    =
    \argmin_{\Pi_{m}} \,
    \sum_{j=1}^{J} \| (I - \Pi_{m}) \hat{\beta}_{j} \|_{2}^{2},
\end{eqnarray*}
where the minimum is taken over all projectors of rank \( m \). One can
easily verify that \( \hat{\Pi} \) projects on the subspace in \( \R^{d}
\) generated by the first \( m \) principal eigenvectors of the matrix \(
\sum_{j} \hat{\beta}_{j} \hat{\beta}_{j}^{\T} \). However, if the number
of informative vectors \( \hat{\beta}_{j} \) is small with respect to \( J
\),  the quality of estimate \( \hat{\Pi} \) can be extremely poor. To
address this drawback of the PCA solution we consider a sparse estimate of
\( \cI \) which uses {\em rounding ellipsoids} for the set
\( \{\hat{\beta}_{j}\}_{j=1}^J \).
\\

For a symmetric positive-definite matrix \( B \) and \( r > 0 \), the ellipsoid
\( \cE_{r}(B) \) is defined as
\[
    \cc{E}_{r}(B)
    =
    \{ x \in \R^{d} \mid x^{\T} B x \le r^{2} \},
\]
For \( \alpha \le 1 \),\( \cE(B) \equiv \cc{E}_{1}(B) \) is \( \alpha
\)-\emph{rounding} ellipsoid for a convex set \( \cc{S} \) if
\[
    \cE_{1/\alpha}(B) \subseteq \cc{S} \subseteq \cE(B).
\]
Note that such ellipsoid exists with \( \alpha = d^{-1/2} \) due to the
Fritz John theorem \cite{John48}. Furthermore, numerically efficient
algorithms for computing \( \sqrt{d} \)-rounding ellipsoids are available,
see e.g. \cite{Nestreov04}. So, for recovering the spatial information
from the vector system \( \{\pm\hat{\beta}_{j}\}_{j=1}^J \) one can
look for the \( d^{1/2} \) rounding ellipsoid for the convex hull \(
\cc{S} \) of points \( \{\pm\hat{\beta}_{j}\}_{j=1}^J \).
\\

We measure the quality of estimation of the subspace \( \cI \) by the
closeness of the estimated projector \( \hat{\Pi} \) to \( \Pi^{*} \):
\begin{eqnarray}\label{error-criterion}
    \varepsilon(\cI,\hat{\cI})= \|\hat{\Pi} - \Pi^{*} \|_{2}^{2}
    =
    \Tr \bigl[ (\hat{\Pi}-\Pi^{*})^{2} \bigr].
\end{eqnarray}
The property of the spatial information recovery, based on the idea of
rounding ellipsoids, is described in the following theorem.
\begin{theorem}
\label{round1}
\par 1. Let \( \cc{S} \) be the convex envelope of the set \( \{\pm
\hat{\beta}_{j}\},\;j=1,...,J \), and let \( \cc{E}_{1}(B) \) be an
ellipsoid inscribed into \( \cc{S} \),  such that \( \cc{E}_{\sqrt{d}}(B)
\) is \( \sqrt{d} \)-rounding ellipsoid for  \( \cc{S} \). Then for any
unit vector \( v \perp \cI \),
\[
    v^{\T} B^{-1} v \le \bm^{2}.
\]
2. If there is \( \mu \in \R^{J} \) with \( \mu_{j}\ge 0 \) and
\( \sum_{j} \mu_{j} = 1 \) such that
\[
    \lambda_{m}\biggl( \sum_{j} \mu_{j} \beta_{j} \beta_{j}^{\T} \biggr)
    \ge
    \lambda^{*}
    >
    2\bm^{2},
\]
where \( \lambda_{m}(A) \) stands for the \( m \)-th principal eigenvalue of \( A \),
then
\begin{eqnarray}
\label{lambdam}
    \lambda_{m}(B^{-1})
    \ge
    \frac{\lambda^{*}-2\bm^{2}}{2\sqrt{d}} \, .
\end{eqnarray}
3. Moreover, let
\( \hat{\Pi} = \hat{\Gamma}_{m} \hat{\Gamma}_{m}^{\T} \) where
\( \Gamma_{m} \) is the matrix of \( m \) principal eigenvectors of \( B^{-1} \). Then
\[
    \|\hat{\Pi}-\Pi^{*}\|_{2}^{2}
    \le
    \frac{4\bm^{2}d\sqrt{d}}{\lambda^{*}-2\bm^{2}}.
\]
\end{theorem}
The proof of the theorem is presented in the appendix.
\\

The results of Theorems \ref{empir1} and \ref{round1} provide a kind of
theoretical justification for the algorithms, presented in the next
section. Indeed, suppose that the test functions \( h_{1},...,h_{L} \) and
the vectors \( \xi_{1},...,\xi_{J} \) are chosen in such a way that there
are at least \( m \) vectors with ''significant'' projection on \( \cI \)
among \( \hat{\beta}_{1},...,\hat{\beta}_{J} \) as in \rf{betaj1}. Then
the projector estimate \( \hat{\Pi} \), computed using the ellipsoid \(
\cE(B) \) which is rounding for the set \( \{\pm\hat{\beta}_{j}\} \), with
high probability will be close to \( \Pi^{*} \).
\\

However, the results about the estimation quality depend critically on the
dimension \( d \). Numerical results also indicate that with growing
dimension, the fraction of non-informative vectors \( \hat{\beta}_{j} \)
increases leading to the situation when some of the longest semi-major
axis of \( \cE_{\sqrt{d}} \) are also non-informative and nearly
orthogonal to \( \cI \). This enforces us to introduce an additional check
of non-normality for the directions suggested by the estimated
ellipsoid \( \cE \).
\\

{\bf Identifying the non-Gaussian subspace by statistical tests:}
Currently the estimation procedure of the vectors \( \beta(\psi_{h,c}) \)
itself does not allow the identification of the semi-axis within the
target space. Hence the basic idea is to apply statistical tests on
normality w.r.t. the significance level \(\alpha \) to the original data
from \( \R^{d} \) projected on every semi-axis of \(\cE_{\sqrt{d}} \). If
the hypothesis of normality is rejected w.r.t. the projected data, the
corresponding semi-axis is used as a basis vector for the reduced target
space \( \cI \).
\\


{\bf Structural adaptation:} At the beginning of the algorithm, we have no
prior information about \( \cI \) and therefore sample the directions \(
\xi_{j} \) and \( \omega_{l} \) randomly from the uniform law. However,
the SNGCA procedure assumes that the obtained estimated structure \(
\hat{\cI} \) delivers some information about \( \cI \) which can be used
for improving the sample mechanism and therefore, the final quality of
estimation. This leads to the \emph{structurally adaptation} iterative
procedure \cite{a3-1}: the step of estimating the vectors \( \{
\hat{\beta}_{j} \}_{j=1}^J \) and the step of estimating subspace \( \cI
\) are iterated, the estimated structural information given by \(
\hat{\cI} \) is used to improve the quality of
estimating the vectors \( \hat{\beta}_{j} \) in the next iteration of SNGCA.
In our implementation, we sample a fraction of directions \( \xi_{j} \)
and \( \omega_{l} \) due to the previously estimated ellipsoid \( \hat{B}
\) and the other part randomly. However the number of the randomly
selected directions remains constant during iteration. In the next section
we present the formal description of SNGCA.

\section{Algorithms}\label{SAlgorithms}

This section describes the principal steps of the procedure. The detailed
description is given in the Appendix.

\subsection{Normalization}

As a preprocessing step the SNGCA procedure uses a
componentwise normalization of the data. Let
\(\sigma=(\sigma_{1},\ldots\sigma_{d}) \) be the standard deviations of the data
components of \( x_{1},\ldots,x_{d} \). For \( i=1,\ldots,N\) the componentwise
normalization of the data is done by \(Y_{i} = \diag(\sigma^{-1}) X_{i}\).

\subsection{Estimation of the vectors from non-Gaussian subspace:}
Let \( \{\omega_{jl}\} \), \( l=1,\ldots,L \), and \( \{\xi_{j}\} \), \(
j=1,\ldots,J \) be two collections of unit vectors called
the measurement directions. Define for all \( j=1,\ldots,J \) and \( l \le
L \), the functions \( h_{jl}(x) = f(x,\omega_{jl}) \), and compute the
vectors \( \hat{\gamma}_{jl} \) and \( \hat{\eta}_{jl} \) due to
(\ref{hathat}). Next, for every \( j \le J \), compute the vector \(
\hat{c}_{j} \) by solving the problem \rf{chat} with \( \xi = \xi_{j} \)
leading to the vector \( \hat{\beta}_{j} \) by \rf{betaj1}.

\subsection{Computing the estimator \( \hat{\Pi} \) of the projector  \( \Pi^{*} \)}
The projector \( \hat{\Pi} \) is constructed on the base of the first \( m
\) principal eigenvectors of the rounding ellipsoid \( \cE \) for the set
\( \cc{S} \) spanned by the vectors \( \pm \hat{\beta}_{j} \), \(
j=1,\ldots,J \). To build the ellipsoid \( \cE \) we use the algorithm in
\cite{Nestreov04} which in fact computes the minimum volume ellipsoid
(MVEE) which covers \( \cc{S} \). For convenience we provide the algorithm
in the appendix.

\subsection{Building the subspace \( \hat{\cE} \) using statistical tests}
In order to construct the projector \( \hat{\Pi} \) the identification of
the \( m \) principal eigenvectors of \( \cE \) that approximate \(\cI\)
is required. In projecting the data onto the semi-axis of \( \cE \) and
testing the projected data on normality the projective approach from the
estimation step is repeated.
\\

Since statistical tests specialized for a certain deviation from the
normal distribution, are more powerful, we use different tests inside of
SNGCA in order to cope with different deviations from normality of the
projected data. To be more precise we use the \( K^{2} \)-test according
to D'Agostino-Pearson \cite{Zar99} to identify a significant asymmetry in
the projected distribution and the EDF-test according to Anderson-Darling
\cite{AnscombeGlynn83} with the modification of Stephens
\cite{Stephens86}, which is sensitive to the tails of the projected
distribution. In order to confirm these test results from above we use the
Shapiro-Wilks test \cite{ShapiroWilks} based on a regression strategy in
the version given by Royston \cite{Royston82a,Royston82c}. Once we have
classified the semi-axis of \(\cE_{\sqrt{d}} \) as being close to the
target space we can use the identified subset of axis in the structural
adaptation step.

\subsection{Structural Adaptation}
The first step of the algorithm assumes that the measurement directions
\( \omega_{jl} \) and \( \xi_{j} \) are drawn randomly from the unit sphere in
\( \R^{d} \).
At each further step of the algorithm
we can use the result of the previous iterations of SNGCA in
order to accumulate information about \( \cI \) in a sequence \(
\hat{\cI}_{1},\hat{\cI}_{2},\ldots \) of estimators of the target space.
This information is used to draw a fraction of the measurement directions from the
estimated subspaces and the other part of such direction is selected randomly.
The procedure is described in detail in algorithm \ref{alg:adaptation}.

\subsection{The stopping criterion}
Suppose that \( \cI \) is a priori given. Then the convergence of SNGCA
can be measured according to the criterion (\ref{error-criterion}). More
precisely we assume convergence if the improvement of the error measured
by (\ref{error-criterion}) from one iteration to the next one is less than
\( \delta \) percent of the error in the former iteration. To this end the
maximum angle \( \theta \) between the subspaces specified by the matrix
of eigenvectors \( V^{(k)}=\big
[\hat{v}_{1}^{(k)},\hat{v}_{2}^{(k)},\ldots \big ] \) and \(
V^{(k+1)}=\big [\hat{v}_{1}^{(k+1)},\hat{v}_{2}^{(k+1)},\ldots \big ] \)
given by
\begin{eqnarray*}
      \cos(\theta)=\max_{x,y}\frac{ | x^{\T} V^{(k)^{\T}} V^{(k+1)}y| }{ \|V^{(k)}x\|_{2} \;\|V^{(k+1)}y\|_{2}}
\end{eqnarray*}
is computed. In the next section we demonstrate the improvement of the
estimation error between subsequent iterations of SNGCA.

\section{Numerical results}\label{numerics}

The aim of this section is to compare SNGCA with other statistical methods
of dimension reduction. The reported results from Projection Pursuit (PP)
and NGCA were already published in \cite{a3-3}.

\subsection{Synthetic Data}\label{synthetic}

Each of the following test data sets includes \( 1000 \) samples in \( 10
\) dimension and each sample consists of \( 8 \)-dimensional independent,
standard and homogeneous Gaussian distributions. The other \( 2 \)
components of each sample are non-Gaussian with variance unity. The
densities of the non-Gaussian components are chosen as follows:
\begin{itemize}
\item[(A)] Gaussian mixture: \( 2 \)-dimensional independent Gaussian mixtures with density of each component given by
\( 0.5\;\phi_{-3,1}(x)+0.5\;\phi_{3,1}(x) \).
\item[(B)] Dependent super-Gaussian: \( 2 \)-dimensional isotropic distribution with density proportional to \( \exp(-\|x\|) \).
\item[(C)] Dependent sub-Gaussian: \( 2 \)-dimensional isotropic uniform with constant positive density for \( \|x\|_{2}\le 1 \) and \( 0 \) otherwise.
\item[(D)] Dependent super- and sub-Gaussian: \( 1 \)-dimensional Laplacian with density proportional to \( \exp(-|x_{Lap}|) \) and \( 1 \)-dimensional
dependent uniform \(  \cU(c,c+1)  \), where \( c=0 \) for \( |x_{Lap}|\le
\log(2) \) and \( c=-1 \) otherwise.
\item[(E)] Dependent sub-Gaussian:\( 2 \)-dimensional isotropic Cauchy distribution with density proportional
to \( \lambda(\lambda^{2}-x^{2})^{-1} \) where \( \lambda=1 \).
\end{itemize}
That means, that the non-normal distributed data are located in a linear subspace. \\

In the sequel we compare SNGCA with PP and NGCA using the test data sets
from above and the estimation error defined in (\ref{error-criterion}).
Each simulation is repeated \( 100 \) times. All simulations are done with
the hyperbolic tangent index as in \rf{tanh}. Since the speed of
convergence varies with the type of non-Gaussian components we use the
maximum number \( maxIter=3\log(d) \) of allowed iterations to stop SNGCA.
In the experiments the error measure \( \epsilon(\cI,\hat{\cI}) \) is used
only to determine the final estimation error. All simulations other than
whose w.r.t. model (C) are computed with a componentwise
pre-normalization.

\begin{figure}[H]
  \centering
  \begin{tabular}{@{} cc@{}}
    \includegraphics[width=0.27\textwidth]{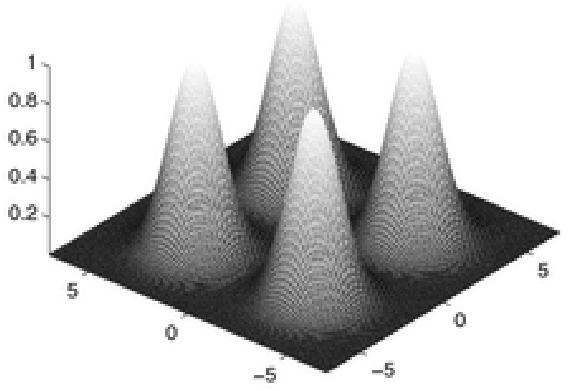} &
    \includegraphics[width=0.27\textwidth]{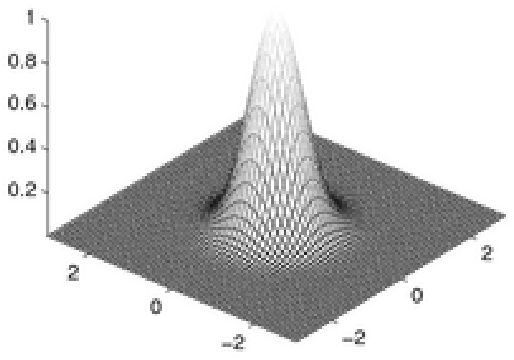}\\
        (A)&(B)\\
    \includegraphics[width=0.27\textwidth]{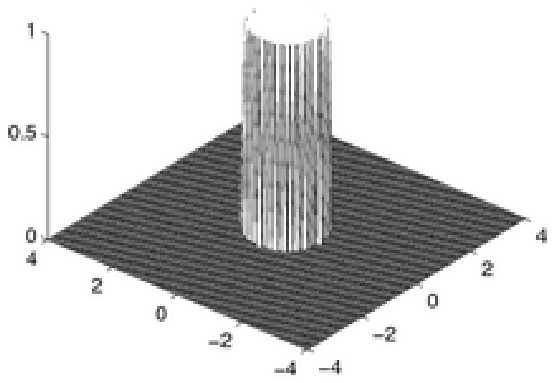} &
    \includegraphics[width=0.27\textwidth]{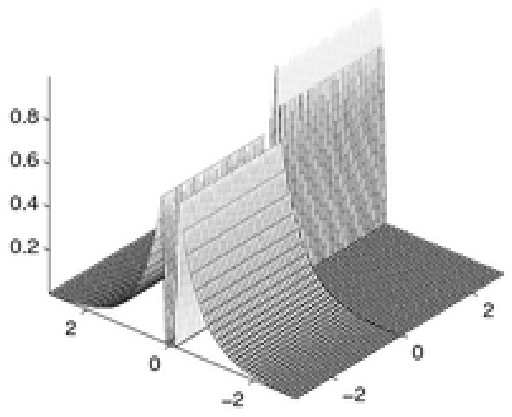}\\
    (C)&(D)\\
    \includegraphics[width=0.27\textwidth]{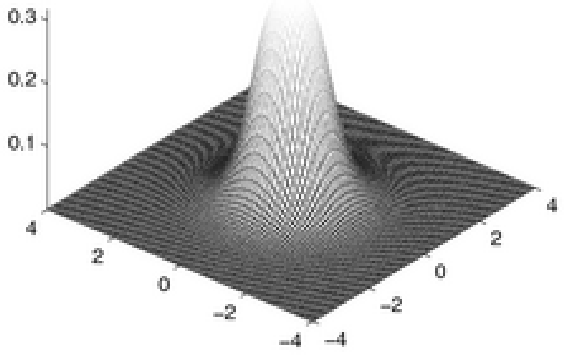} &\\
    (E)& \\
  \end{tabular}
  \caption{densities of the non-Gaussian components: (A) \( 2 \)d independent Gaussian mixtures, (B) \( 2 \)d isotropic
    super-Gaussian, (C) \( 2 \)d isotropic uniform and (D) dependent \( 1 \)d Laplacian with additive \( 1
    \)d uniform, (E) \( 2 \)d isotropic sub-Gaussian  \label{fig:density}
  }
\end{figure}

Figure \ref{fig:density} illustrates the densities of the non-Gaussian
components of the test data. For all numerical experiments reported in
this article the dimension of the target space \( \cI \) is a priori given
as a tuning parameter for the algorithm.\\

Since the optimizer used in PP tends to trap in a local minima in each of
the 100 simulations, PP is 10 times restarted with random starting points.
The best result w.r.t. (\ref{error-criterion}) is reported as the result
of each PP-simulation. In all PP-simulations the number of non-Gaussian
dimensions is a priori given. In the next figure \ref{comp_plots} we
present boxplots of the error (\ref{error-criterion}) obtained from the
methods PP, NGCA and SNGCA.

\begin{figure}[H]
\begin{center}
  \begin{tabular}{@{} cc@{}}
    \includegraphics[width=0.45\textwidth]{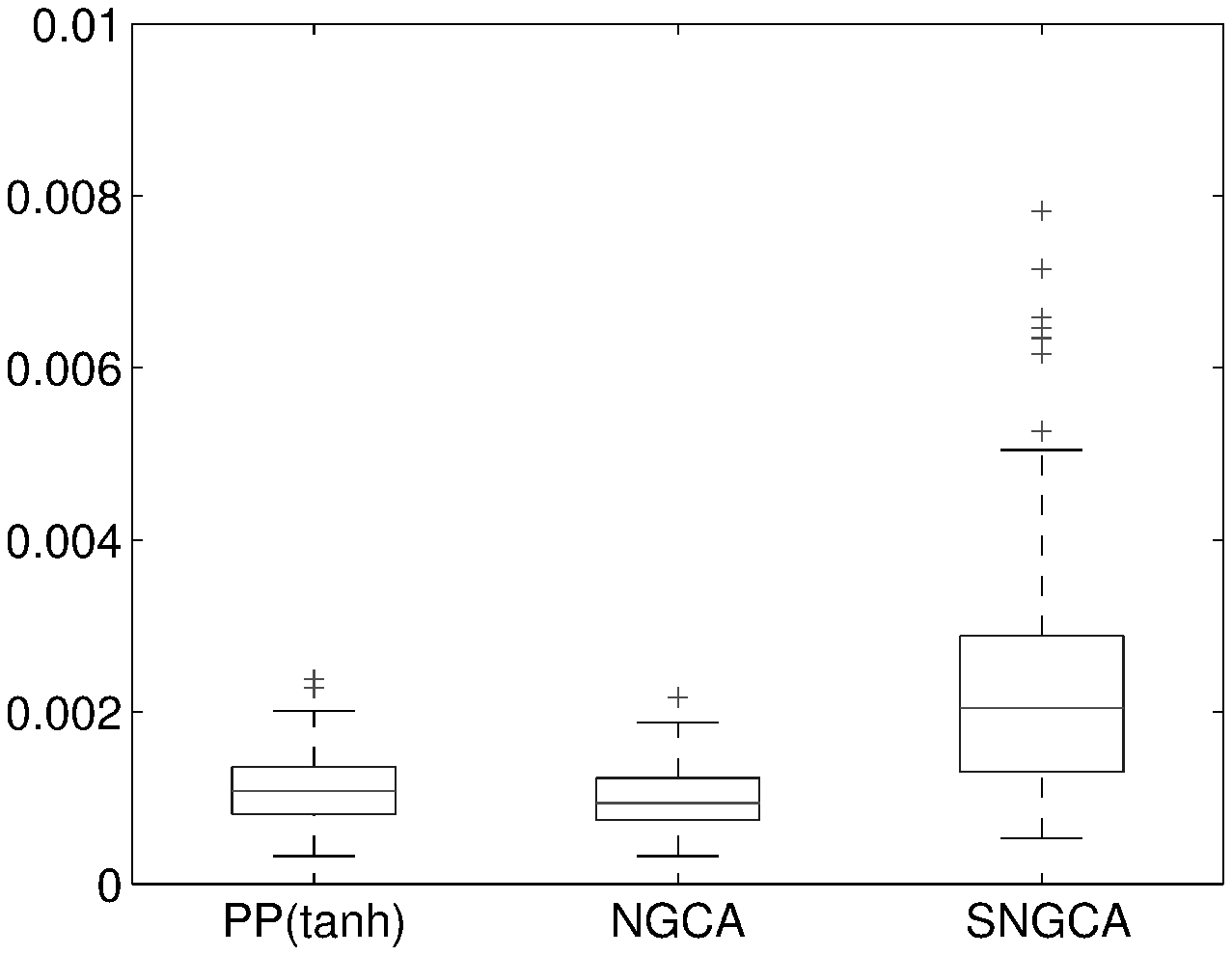}&
    \includegraphics[width=0.45\textwidth]{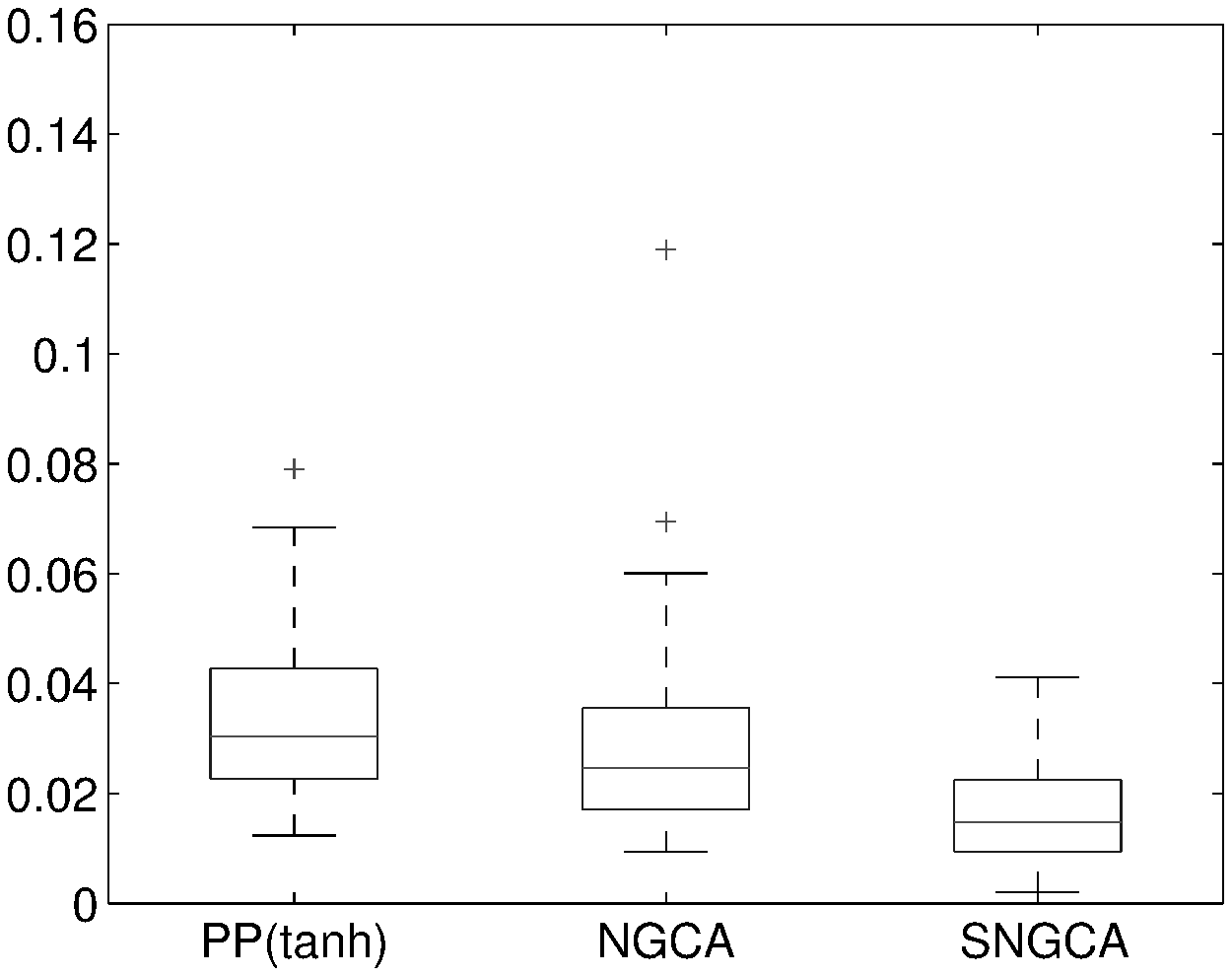}\\
        (A)&(B)\\
        \includegraphics[width=0.45\textwidth]{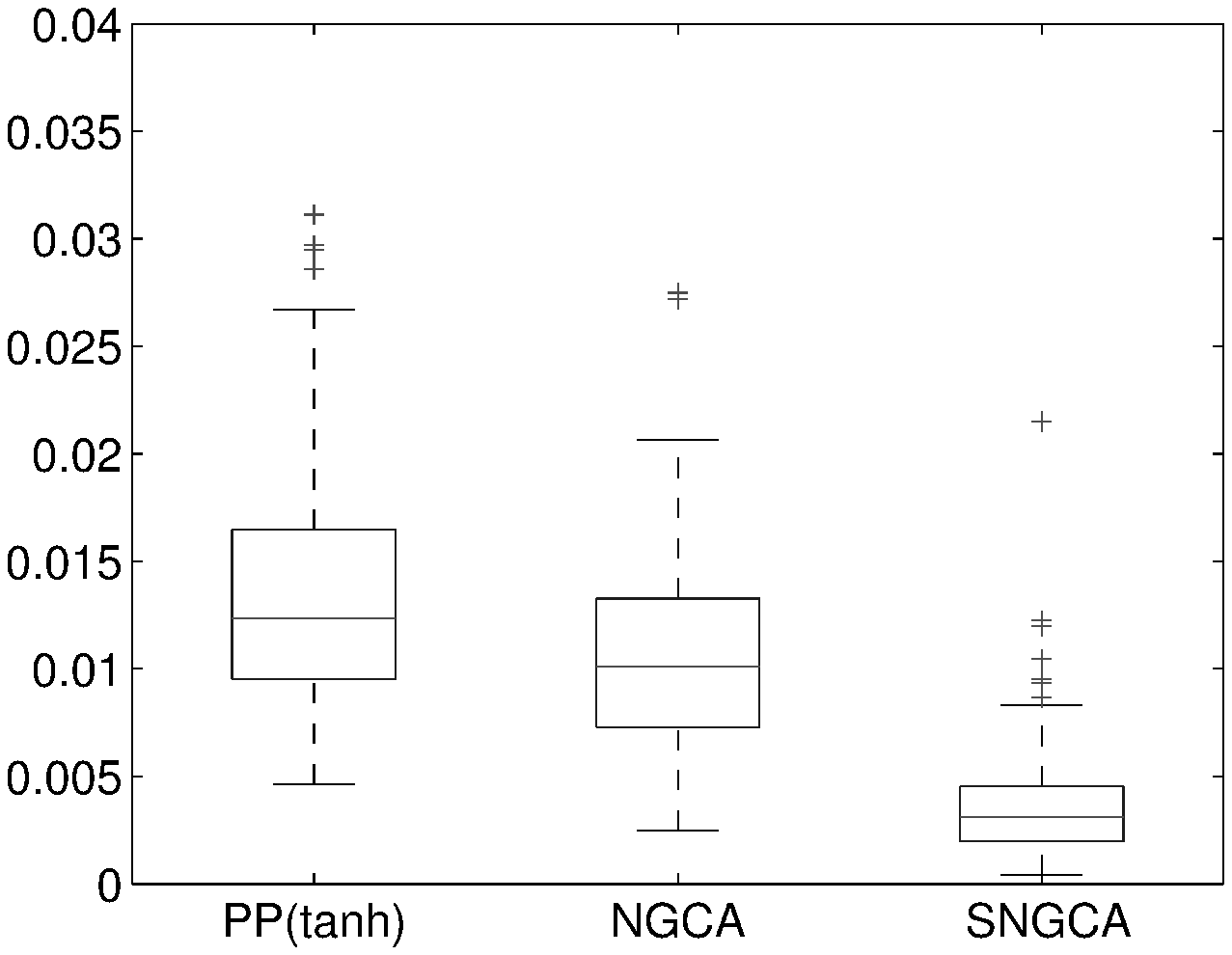}&
        \includegraphics[width=0.45\textwidth]{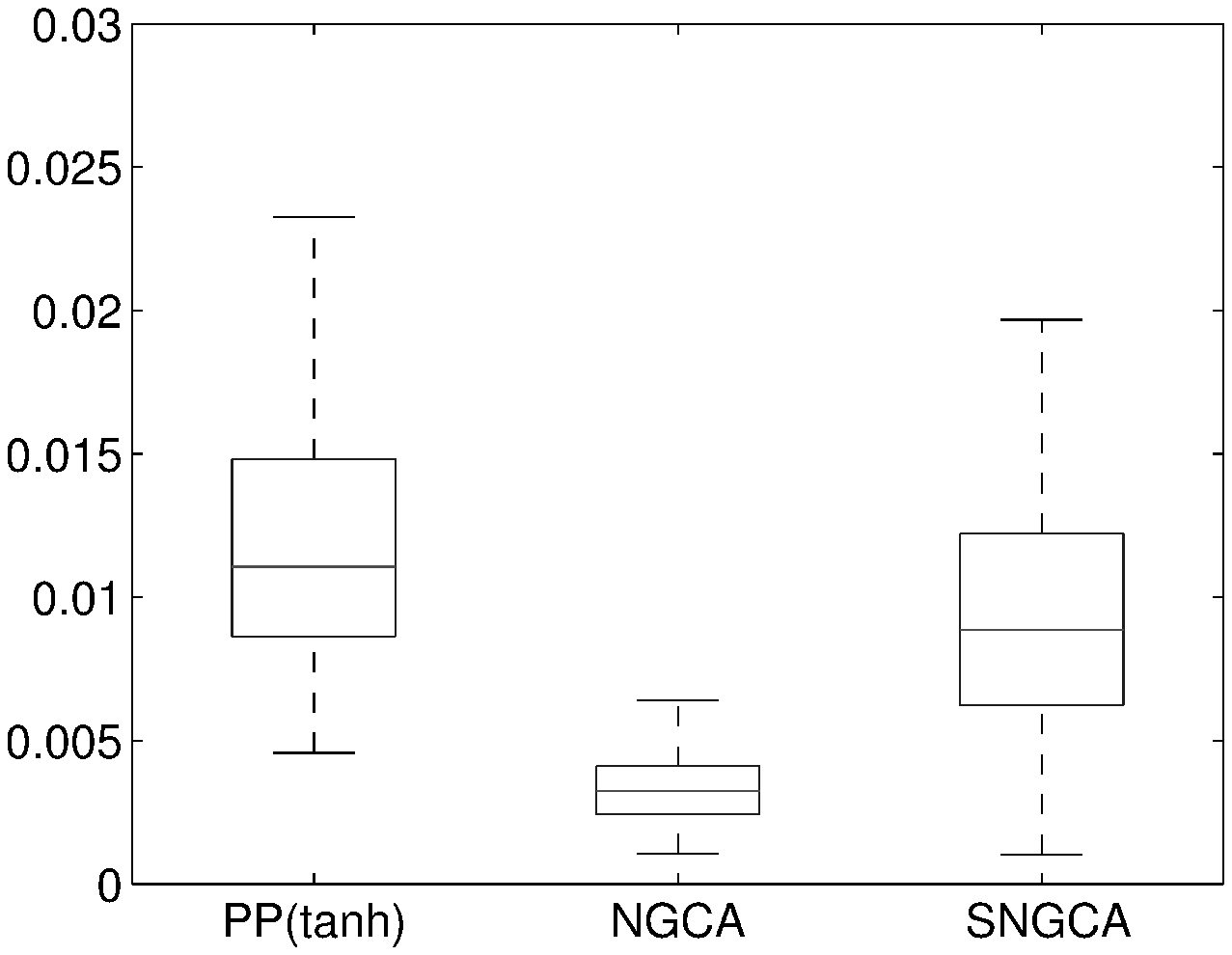}\\
    (C)&(D)\\
        \includegraphics[width=0.45\textwidth]{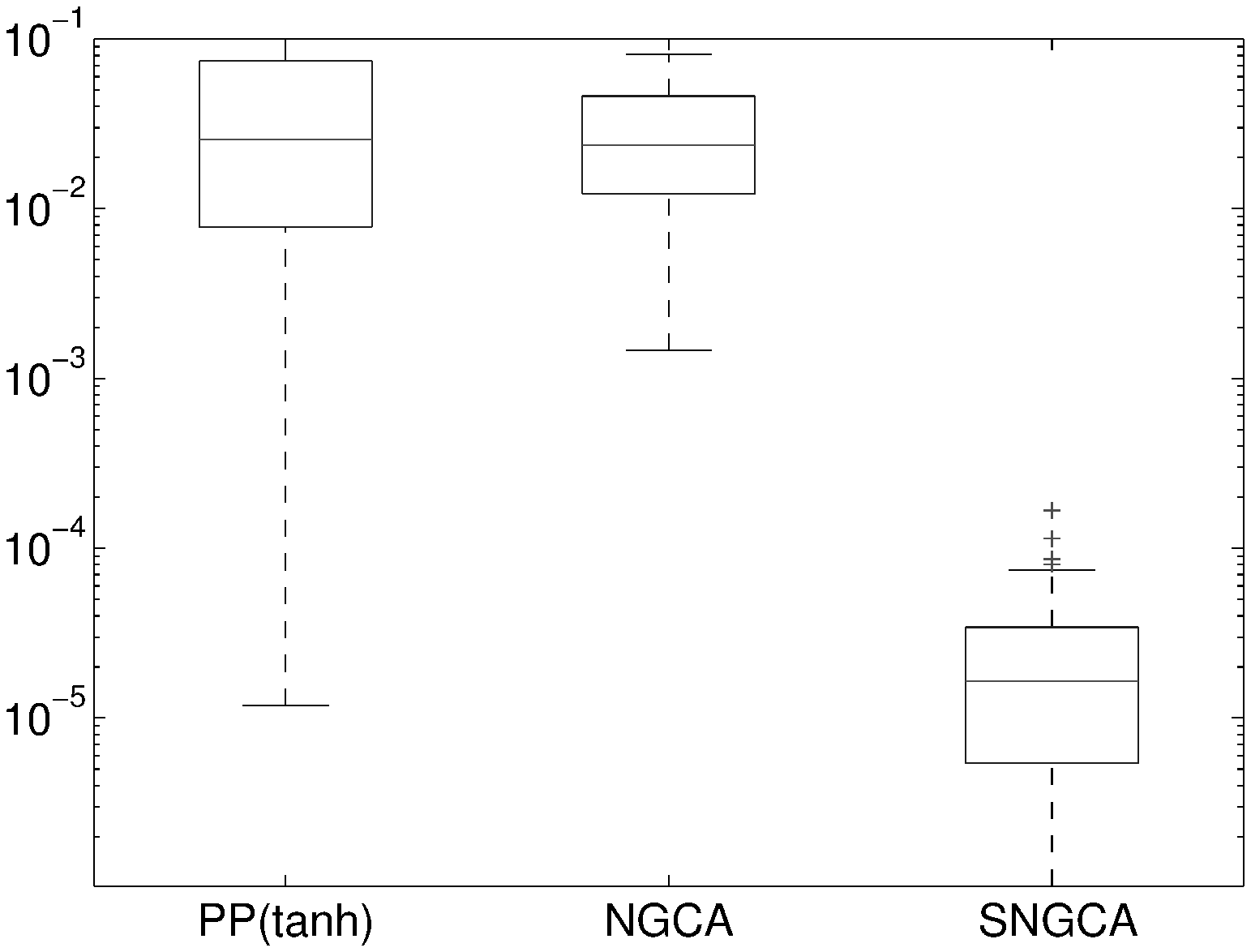}& \\
    (E)& \\
  \end{tabular}
\caption{performance comparison in \( 10 \) dimensions of PP and NGCA versus SNGCA (wrt.
the error criterion \( \cE(\widehat{\cI},\cI) \) ) using the index \( tanh(x) \). The
doted line denotes the mean, the solid lines the variance of (\ref{error-criterion}).
\label{comp_plots}}
\end{center}
\end{figure}
Concerning the results of SNGCA on the data sets (A) and (D) we observe a
slightly inferior performance compared to NGCA. In case of model (A) this
is due to the fact that most of the data projections have almost a
Gaussian density. Consequently the decrease of the estimation error is
slow with increasing number of iterations. In case of the model (D) the
higher variance of the results indicate that the initial sampling of the
data sets gives a poor result. Consequently more iterations are needed to
get an estimation error that is comparable to the result of NGCA. In order
to illustrate this interpretation we report in table (\ref{tab:progress})
the progress of SNGCA w.r.t. estimation error
\(\varepsilon(\cI,\hat{\cI}) \) in each iteration for every test model.\\

\begin{table}[h]
\begin{center}
\begin{tabular}{@{} cc@{}}
\begin{tabular}{|c|c|c|} \hline
  \( j \) & \( \mu_{\epsilon} \) & \( \sigma^2_{\epsilon} \) \\ \hline
  1 & 0.232504 & 0.045787 \\
  2 & 0.163022 & 0.072263 \\
  3 & 0.066537 & 0.032436 \\
  4 & 0.009380 & 0.021975 \\
  5 & 0.002359 & 0.000853 \\ \hline
\end{tabular} &
\begin{tabular}{|c|c|c|} \hline
  \( j \) & \( \mu_{\epsilon} \) & \( \sigma^2_{\epsilon} \) \\ \hline
  1 & 0.30350 &  0.175313 \\
  2 & 0.144430 & 0.057856 \\
  3 & 0.088142 & 0.015168 \\
  4 & 0.041420 & 0.008197 \\
  5 & 0.026436 & 0.000917 \\ \hline
\end{tabular} \\
(A) & (B) \\[2ex]
\begin{tabular}{|c|c|c|} \hline
  \( j \) & \( \mu_{\epsilon} \) & \( \sigma^2_{\epsilon} \) \\ \hline
  1 & 0.040556 &  0.004215\\
  2 & 0.016012 &  0.002441\\
  3 & 0.012427 &  0.001105\\
  4 & 0.008874 &  0.000169\\
  5 & 0.003770 &  0.000125\\ \hline
\end{tabular}  &
\begin{tabular}{|c|c|c|} \hline
  \( j \) & \( \mu_{\epsilon} \) & \( \sigma^2_{\epsilon} \) \\ \hline
  1 & 0.203419 & 0.044672 \\
  2 & 0.023023 & 0.000314 \\
  3 & 0.019960 & 0.000211 \\
  4 & 0.012709 & 0.000197 \\
  5 & 0.009343 & 0.000127 \\ \hline
\end{tabular}  \\
(C) & (D) \\ [2ex]
\begin{tabular}{|c|c|c|} \hline
  \( j \) & \( \mu_{\epsilon} \) & \( \sigma^2_{\epsilon} \) \\ \hline
  1 & 0.2762e-3 & 0.1371e-6 \\
  2 & 0.0450e-3 & 0.0031e-6 \\
  3 & 0.0416e-3 & 0.0033e-6 \\
  4 & 0.0360e-3 & 0.0014e-6 \\
  5 & 0.0287e-3 & 0.0024e-6 \\ \hline
\end{tabular}  &   \\
(E) &  \\
\end{tabular}
\caption{Progress of SNGCA for test models in \( 10 \) dimensions with increasing number
\( j \) of iterations. The empirical mean of the error \( \cE(\widehat{\cI},\cI) \)
defined in (\ref{error-criterion}) is denoted by \( \mu_{\epsilon} \) and \(
\sigma^2_{\epsilon} \) is its empirical variance.\label{tab:progress}}
\end{center}
\end{table}

{\bf Illustration of one-step-improvement:} We shall now illustrate the
iterative gain of information about the EDR space. To this end we use the
projection of \( \hat{\beta}_{j} \) to the EDR-space in order to
demonstrate, how the algorithm works. Figure \ref{fig:adaptation} shows that \( dist(\hat{\beta},\hat{\cI}) \)
decreases with increasing number of iterations. We observe,
that estimators \( \hat{\beta} \) with higher norm tend to be close to \(
\cI \). Nevertheless, this can not be assured for much higher dimensions.
Moreover the improvement in each iteration depends on
the size of the sampling of measurement directions. 

\begin{figure}[H]
    \centering
    \includegraphics[width=0.56\textwidth]{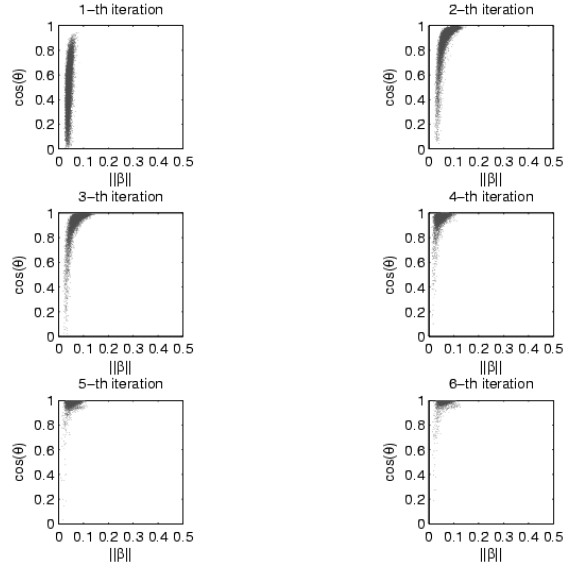}
    \caption{illustrative plots of SNGCA applied to toy 20 dimensional data of type {\bf(C)}
    (see section \ref{numerics}): We show \( \| \hat{\beta}\|  \) vs. \(
    \cos(\theta(\hat{\beta},\cI)) \) for different iterations of the algorithm where \(
    \cI \) is the a priori known EDR-space.} \label{fig:adaptation}
\end{figure}

\begin{figure}[H]
\begin{center}
  \begin{tabular}{@{} cc@{}}
    \includegraphics[width=0.3\textwidth]{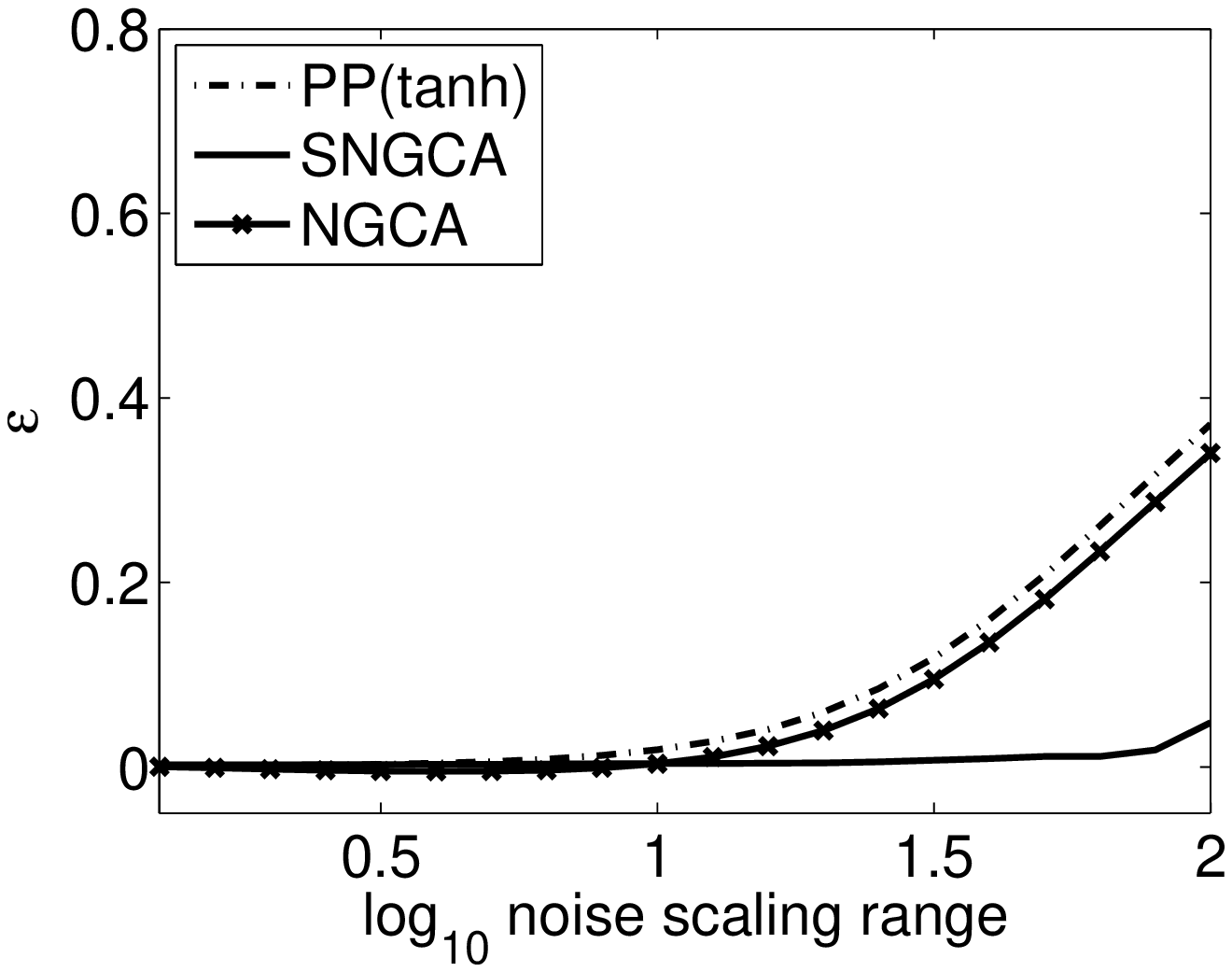}&
    \includegraphics[width=0.3\textwidth]{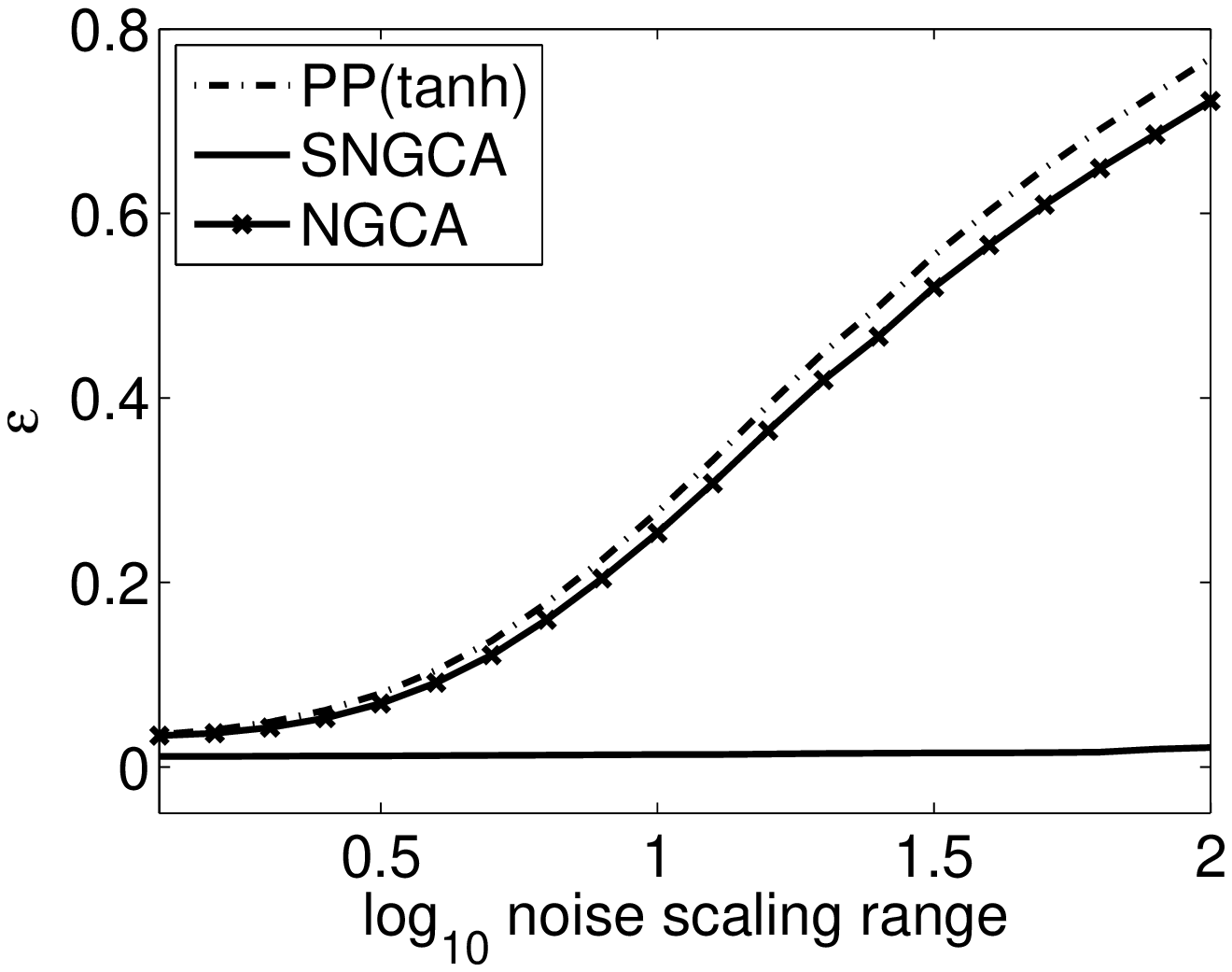}\\
        (A)&(B)\\
        \includegraphics[width=0.3\textwidth]{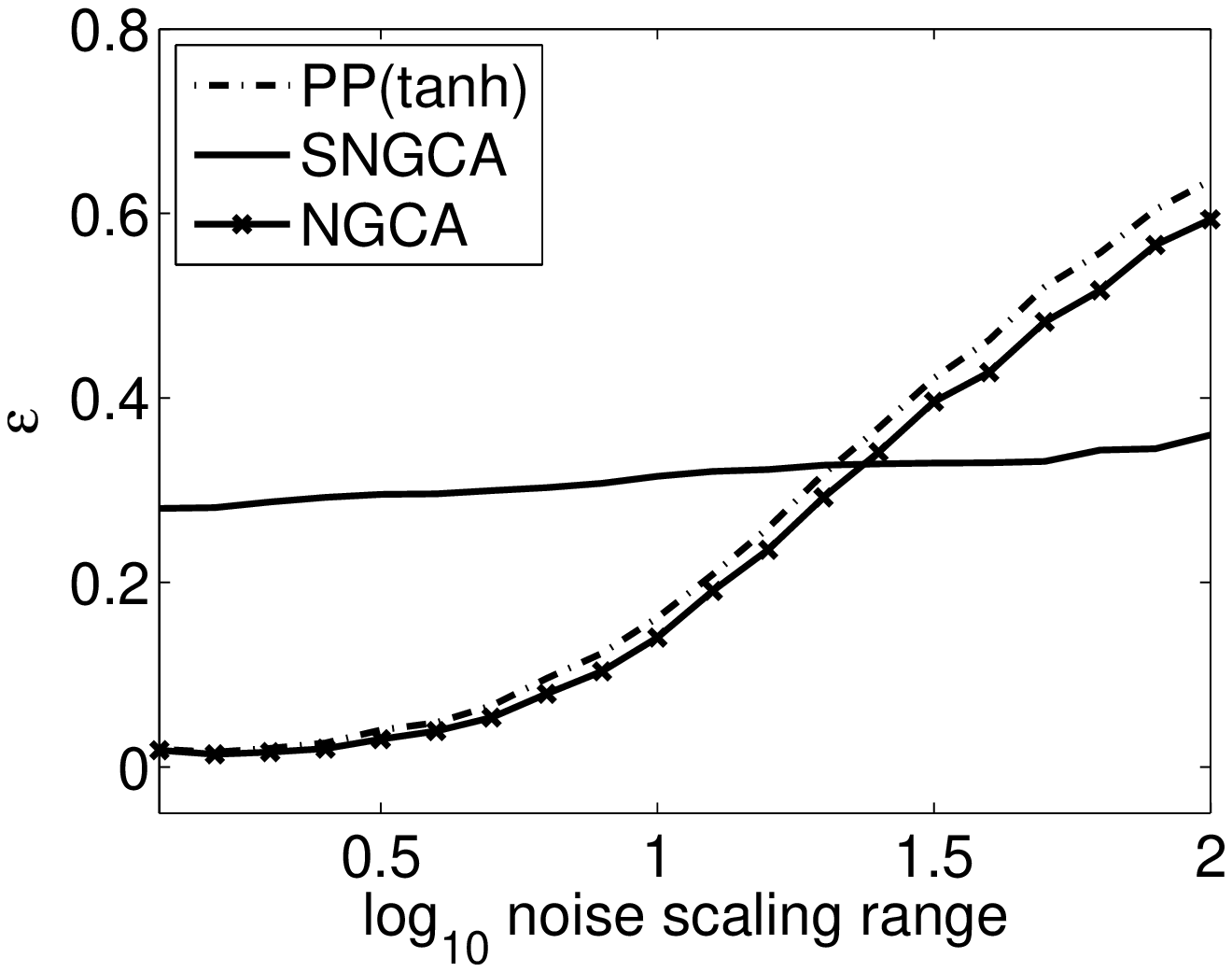}&
        \includegraphics[width=0.3\textwidth]{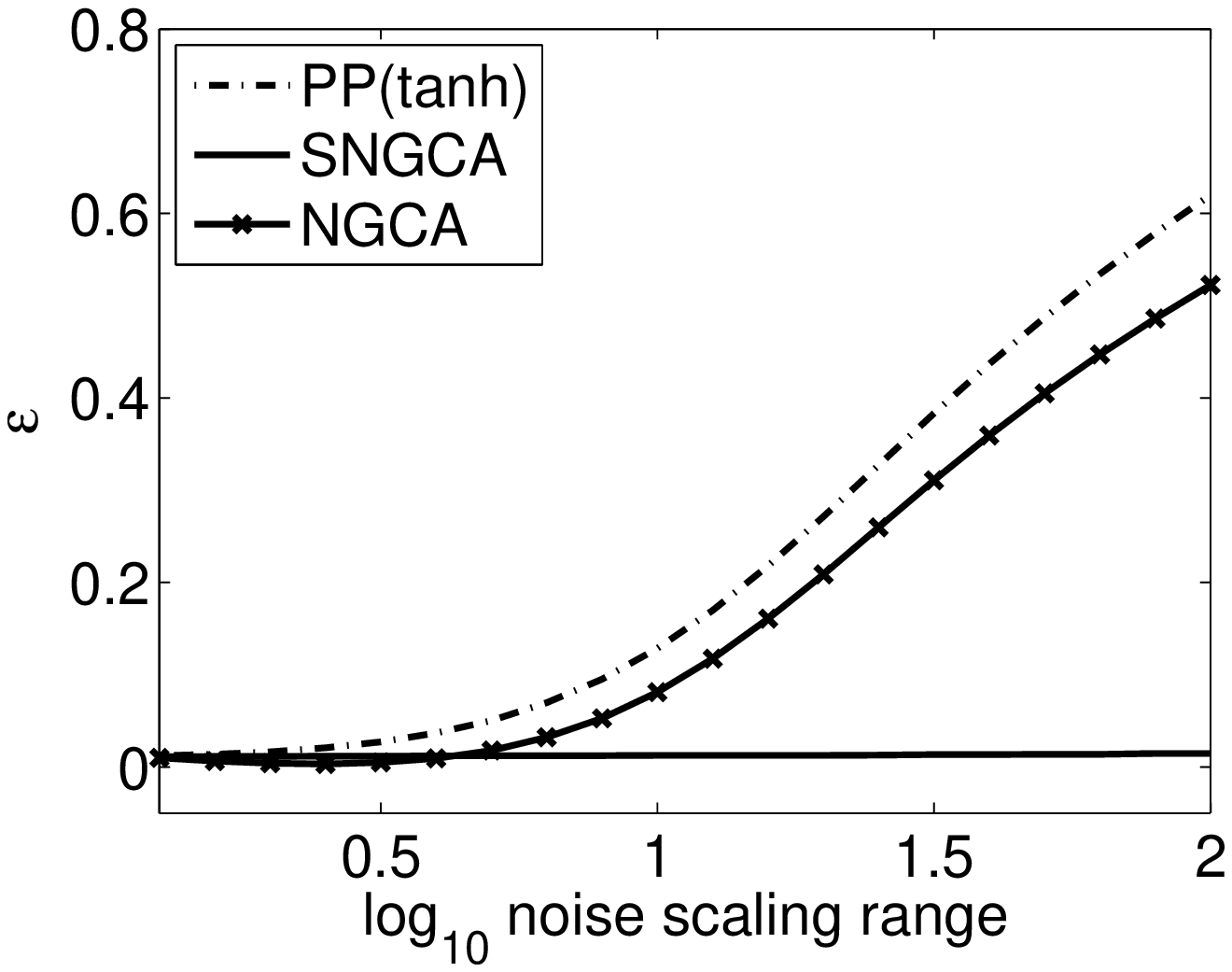}\\
    (C)&(D)\\
        \includegraphics[width=0.3\textwidth]{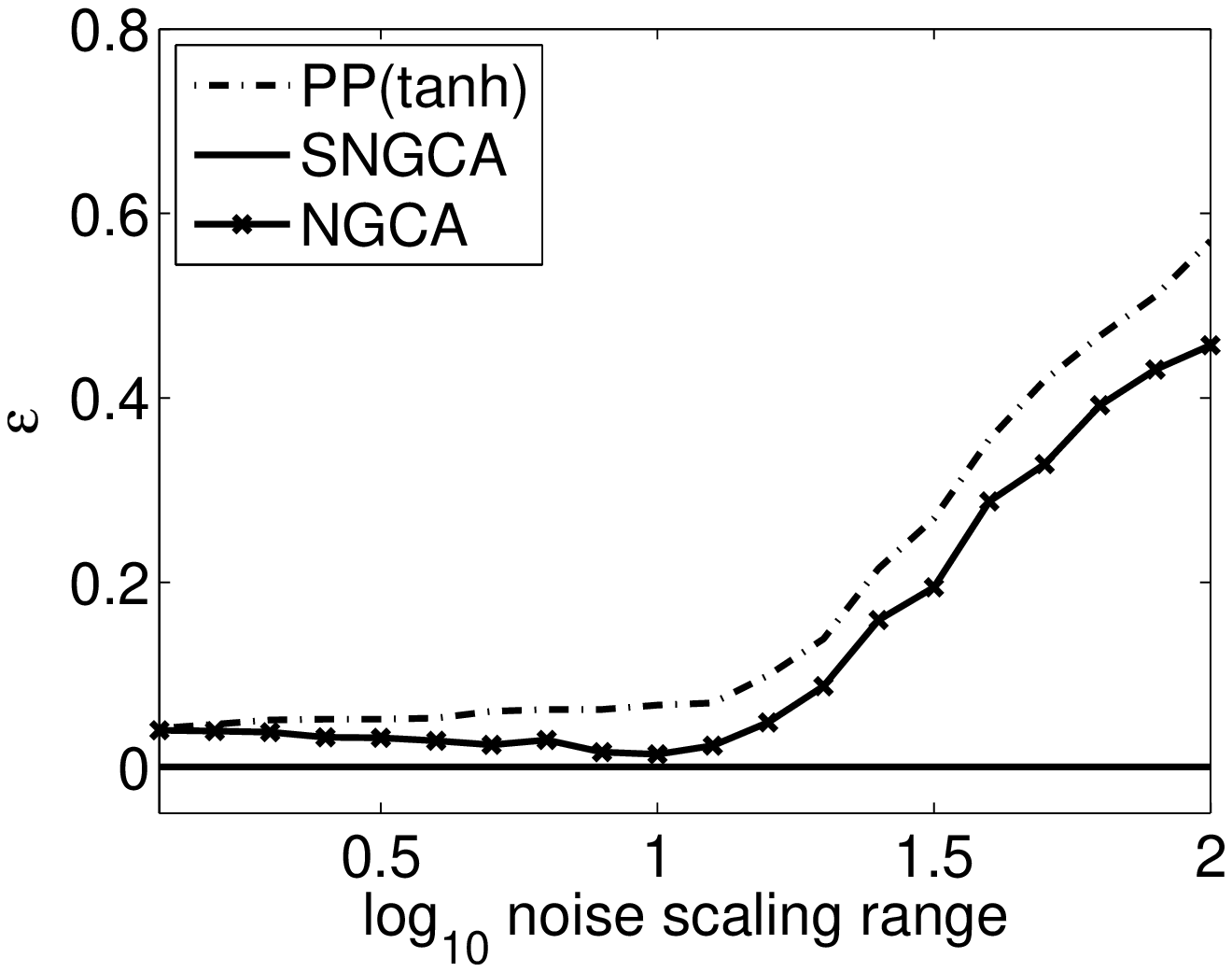}& \\
    (E)& \\
  \end{tabular}
    \caption{results wrt. \( \cE(\widehat{\cI},\cI) \) with deviations of Gaussian components
    following a geometrical progression on \( [10^{-r},10^r] \) where \( r \) is the parameter
    on the abscissa) \label{fig:noise}.}
\end{center}
\end{figure}
Now let us switch to the question of robustness of the estimation
procedure with respect to a bad conditioning of the covariance matrix \(
\Sigma \) of the data. In figure \ref{fig:noise} we consider the same test
data sets as above. The non-Gaussian coordinates always have unity
variance, but the standard deviation of the \( 8 \) Gaussian dimensions
now follows the geometrical progression \(
10^{-r},10^{-r+2r/7},\ldots,10^r \) where \( r=1,\ldots,8 \). Again we
apply a componentwise normalization procedure to the data from the models
(A), (B), (D), (E). We observe that the condition of the
covariance matrix heavily influences the estimation error for the methods
NGCA and PP(tanh). In comparison SNGCA is independent of differences in
the noise variance along different directions in most cases. Only the
detection of the uniform distribution by SNGCA is influenced by the
condition of \( \Sigma \).\\
\begin{figure}[H]
\begin{center}
  \begin{tabular}{@{} cc@{}}
    \includegraphics[width=0.4\textwidth]{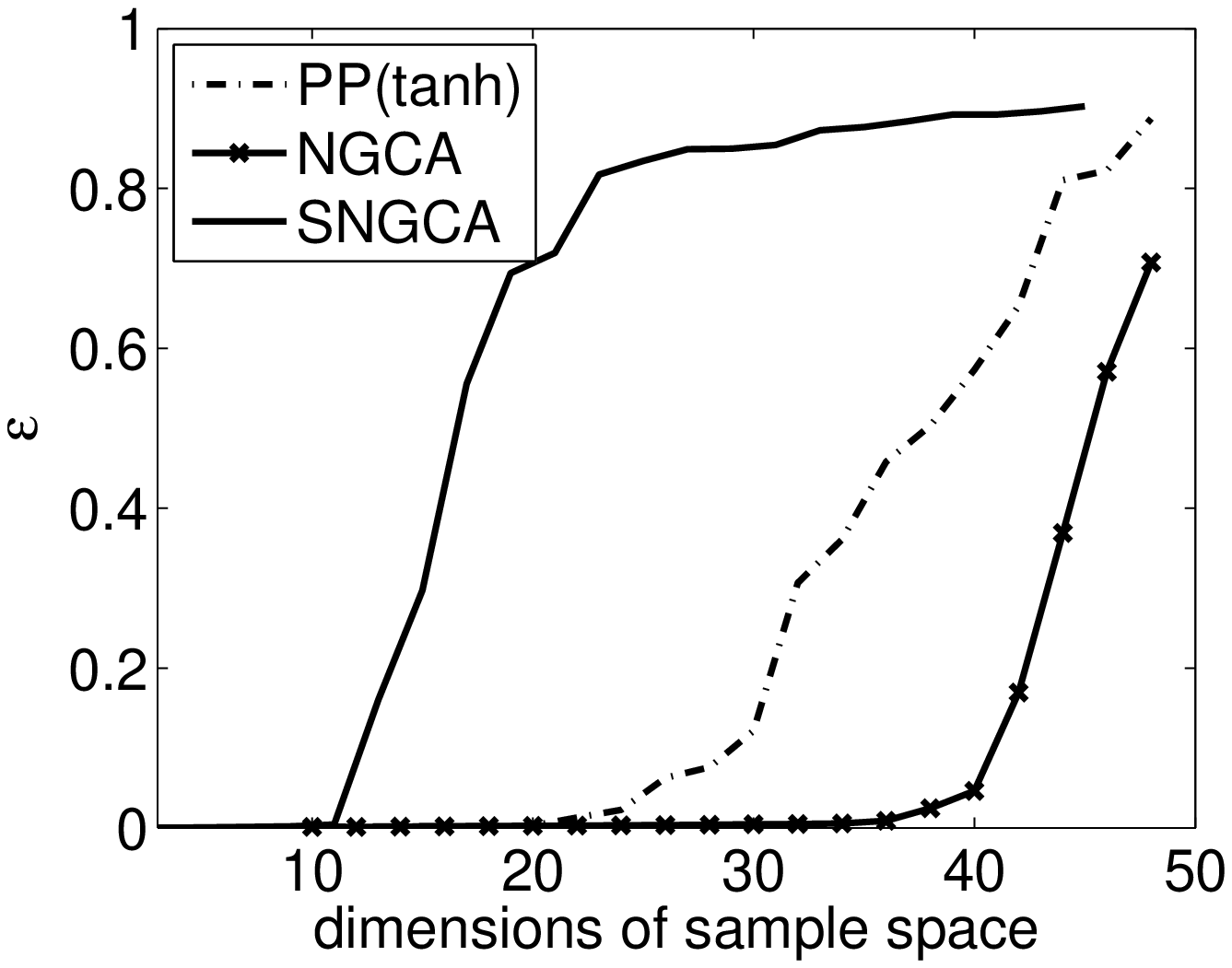}&
    \includegraphics[width=0.4\textwidth]{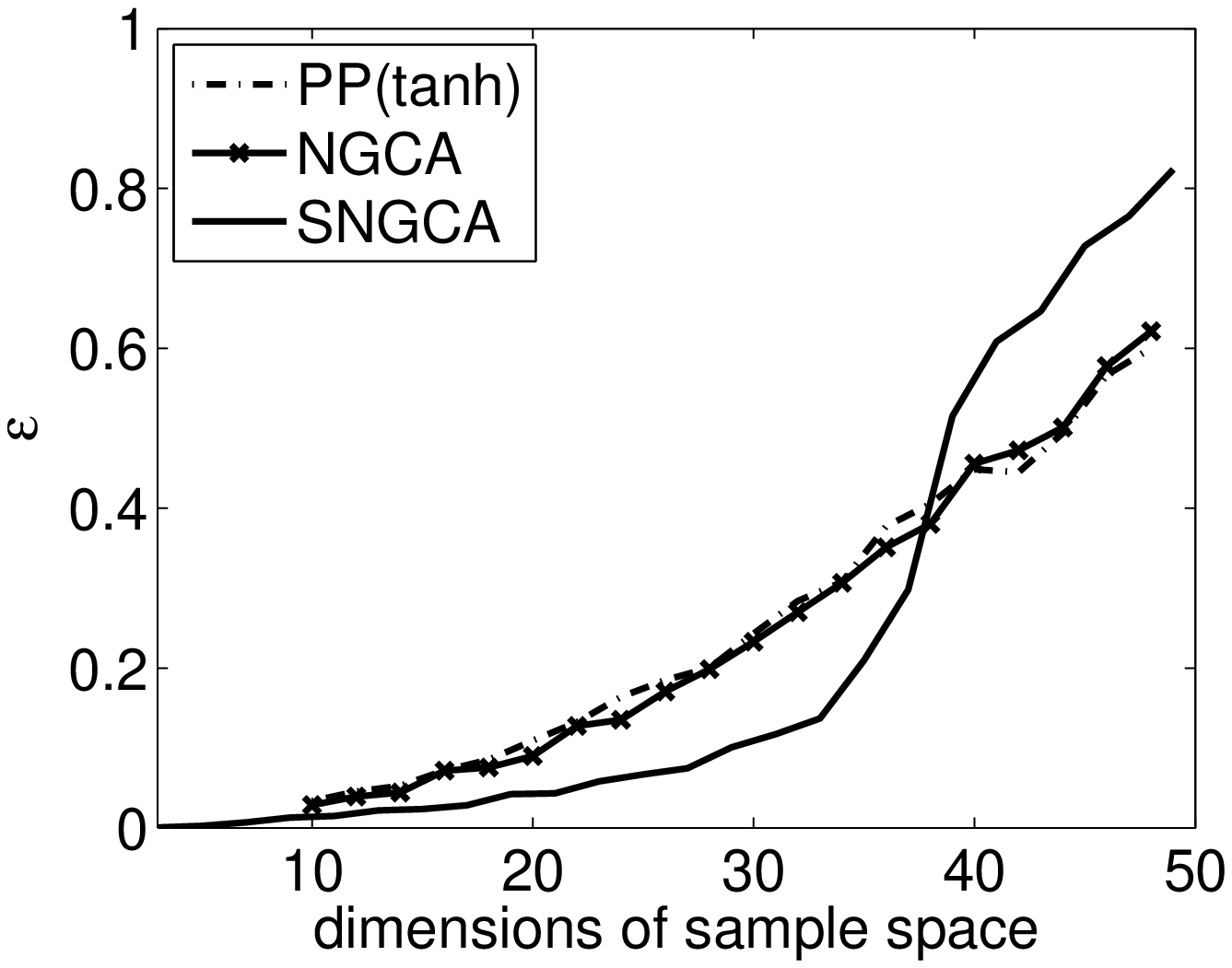}\\
        (A)&(B)\\
        \includegraphics[width=0.4\textwidth]{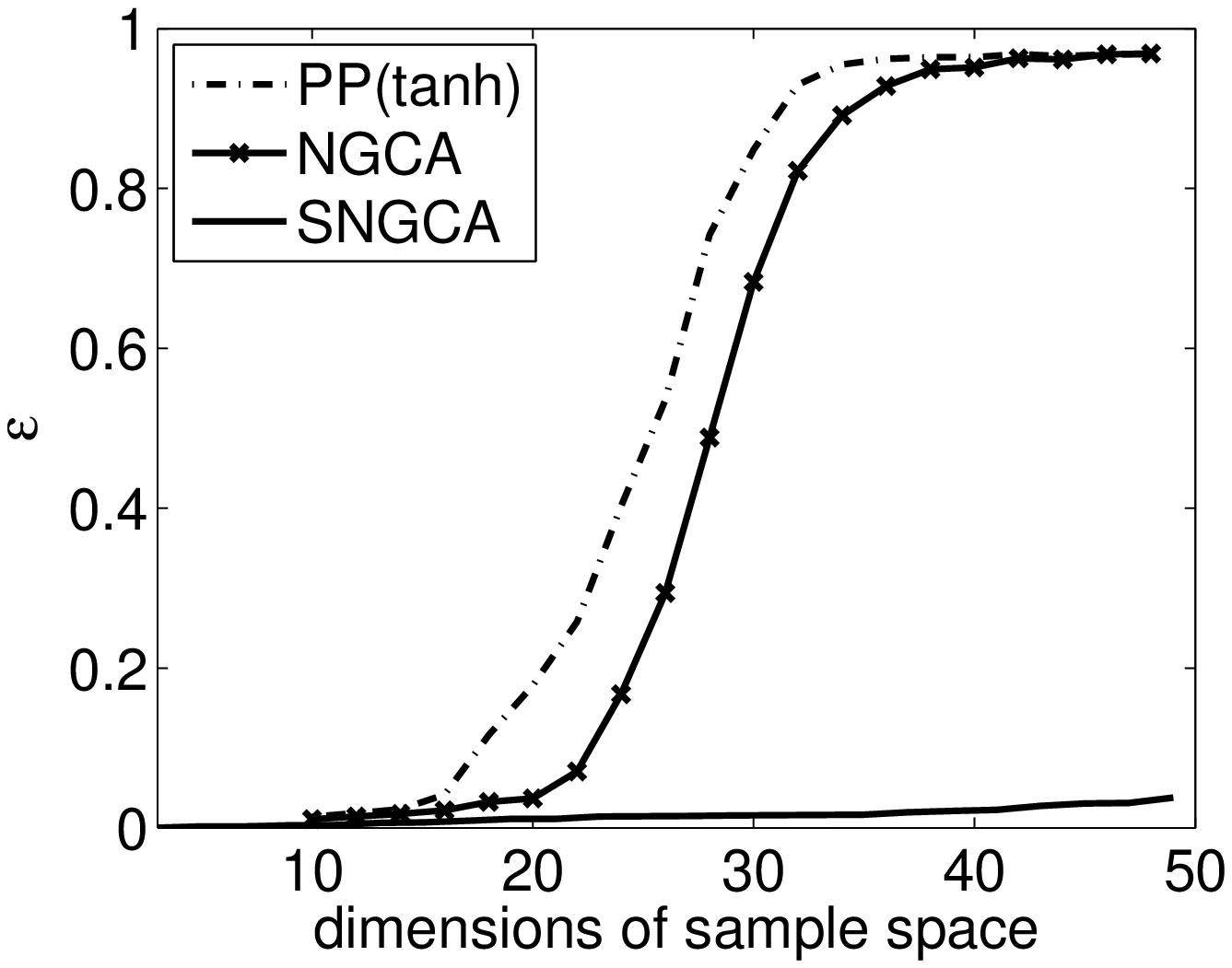}&
        \includegraphics[width=0.4\textwidth]{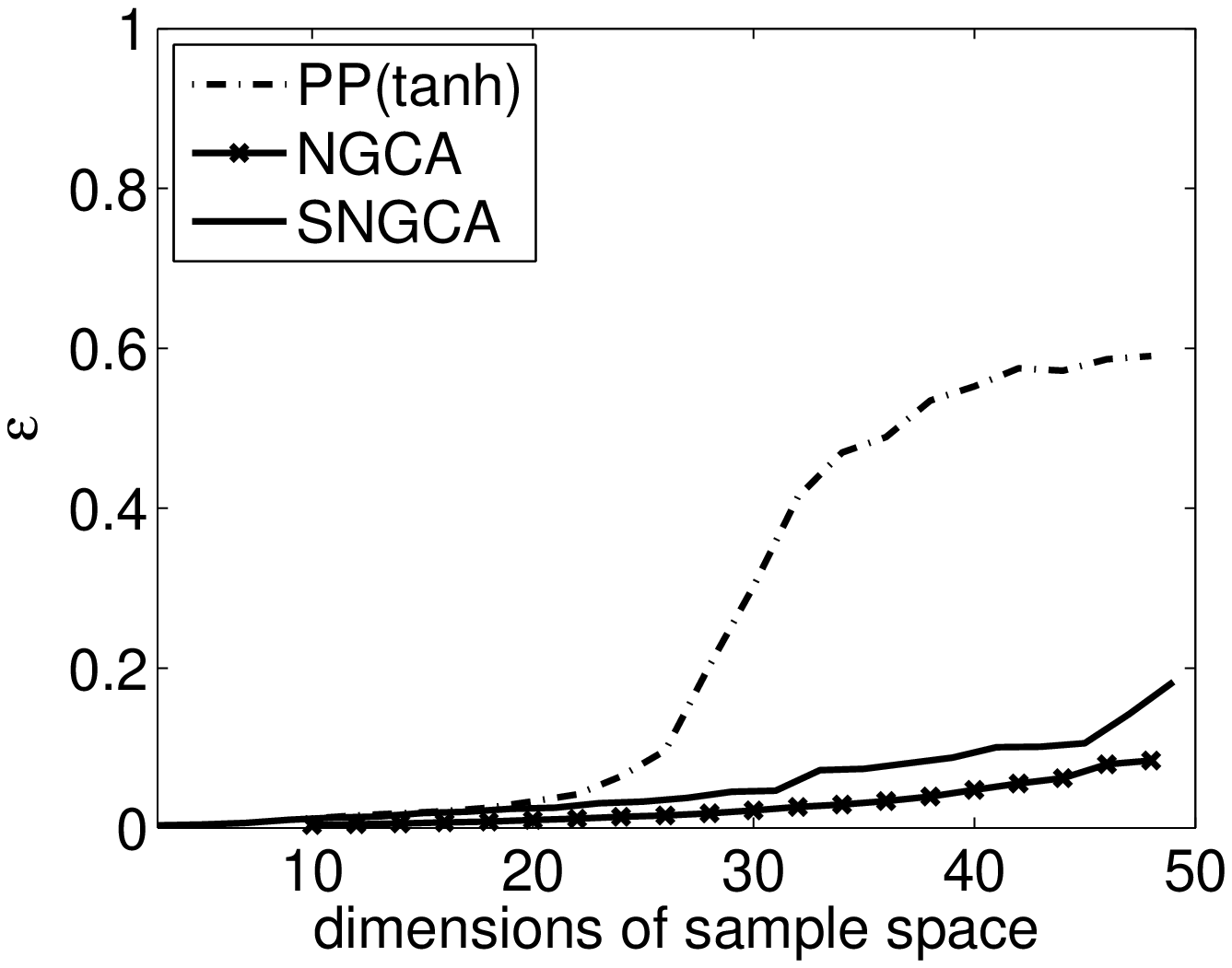}\\
    (C)&(D)\\
        \includegraphics[width=0.4\textwidth]{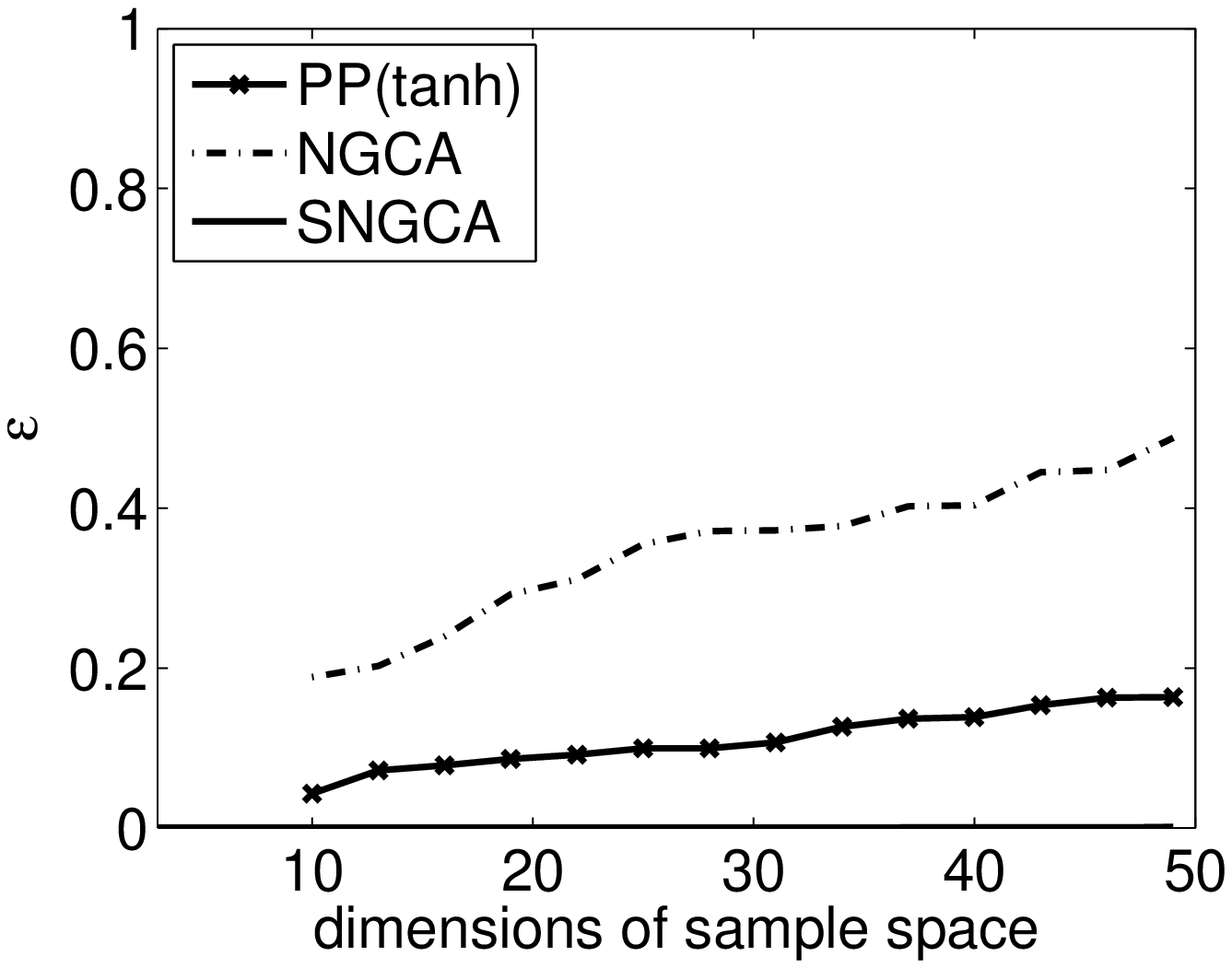}& \\
    (E)& \\
  \end{tabular}
    \caption{results wrt. \( \cE(\widehat{\cI},\cI) \) with increasing number of gaussian components.
    \label{fig:dim-increase}}
\end{center}
\end{figure}

Figure \ref{fig:dim-increase} compares the behavior of SNGCA with PP and
NGCA as the number of standard and homogeneous Gaussian dimensions
increases. As described above we use the test models with \( 2
\)-dimensional non-Gaussian components with unity variance. We plot the
mean of errors \( \varepsilon(\hat{\cI},\cI) \) over \( 100 \) simulations
w.r.t. the test models (A) to (E).\\

Again concerning the mean of errors \( \varepsilon(\hat{\cI},\cI) \) over
\( 100 \) simulations of PP and NGCA we find a transition in the error
criterion to a failure mode for the test models (A), (C) between \( d=30
\) and \( d=40 \) and between \( d=20 \) and \( d=30 \) respectively. For
the test models (B),(D) and (E) we found a relative continuous increase in
\( \varepsilon(\hat{\cI},\cI) \) for the methods PP and NGCA. In
comparison SNGCA fails to analyze test model (A) independently from the
size of the sampling, if the dimension exceeds \( d=12 \). Concerning test
model (B) there is a sharp transition in the simulation result
between \( d=35 \) and \( d=40 \).\\

{\bf Failure modes:} In order to provide a better insight into the details
of the failure modes we present box plots of \( \varepsilon(\hat{\cI},\cI)
\) in the transition phases w.r.t. the models (A) and (B).\\
\begin{figure}[H]
 \begin{center}
    \includegraphics[width=0.8\textwidth]{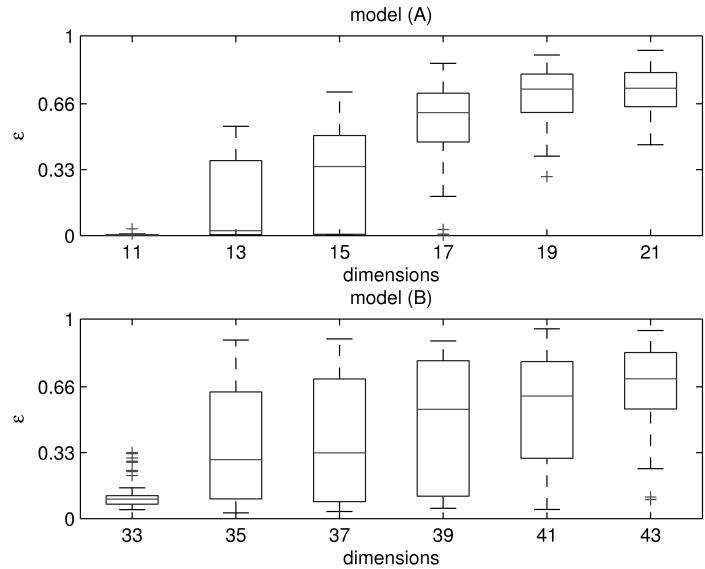}
  \caption{failure modes of SNGCA - upper figure: model (A) - lower figure: model(B)}
  \label{fig:failureModes}
 \end{center}
\end{figure}
Figure \ref{fig:failureModes} demonstrates the differences in the
transition phases of model (A) and (B) respectively. The transition phase
is characterized by a high variance of the estimation error. For model (A)
the increase of the variance \( \sigma^{2}_{\varepsilon} \) of \(
\varepsilon(\hat{\cI},\cI) \) beginning at dimensions \( 13 \) and its
decrease beginning at dimension \( 15 \) indicates that a sharp transition
phase happens in the interval \( [13,15] \). For higher dimensions more
iterations of SNGCA have a decreasing effect on the estimation result.
This indicates that by the sampling of the measurement directions, we can
not detect the non-Gaussian components of the data density. For model (B)
the transition phase starts at dimension \( 35 \) and ends at dimension \( 43 \).\\

Moreover the decrease of \( \sigma^{2}_{\varepsilon} \) towards higher
dimensions and the increase of the mean of \( \varepsilon(\hat{\cI},\cI)
\) is much slower. This indicates that the non-Gaussian density components
might be detectable if we would allow much more iterations of SNGCA and an
enlarged size of the set of measurement directions. This observation
motivates the interpretation that the Monte-Carlo sampling of the
measurement directions is a very poor strategy that fails to provide
sufficient information about the Laplace distribution in high dimensions.
Currently the SNGCA performance is limited by the sampling strategy.

\subsection{Application to real life examples}\label{real}

We consider a simulating of a mixture of oil and gas flowing under high
pressure through a pipeline. Under these physical conditions different
phases of the oil-gas-mixture may exist at the same time in the phase
space \( \Gamma \). Only some of these phase configurations in \( \Gamma
\) are stable over long periods of time. Consequently one expects some
clusters of points in \( \Gamma \) indicating the physical state of the
mixture. The \( 12 \)-dimensional data set, obtained by numerical
simulations of a stationary physical model, was already used before for
testing techniques of dimension reduction \cite{GTM98}. The data set comes
with a subset of training data and a subset of test data. The length of
the time series is \( 1000 \) in each dimension. \\

The task with this data is to find the clusters representing the stable
configurations in the training data set. It is not known a priori if some
dimensions are more relevant than others. However it is known a priori
that the data is divided into \( 3 \) classes, indicated by different
shapes of the data points. The cluster information is not used in finding
the EDR-space. Again we compare SNGCA with NGCA and PP using the
hyperbolic tangent index \rf{tanh}. For PP and NGCA the results are shown
in figure \ref{fig-oildataPP}. They were already published in \cite{a3-3}.

\begin{figure}[h]
  \begin{center}
   \begin{tabular}{@{} cccc@{}}
    \includegraphics[width=0.45\textwidth]{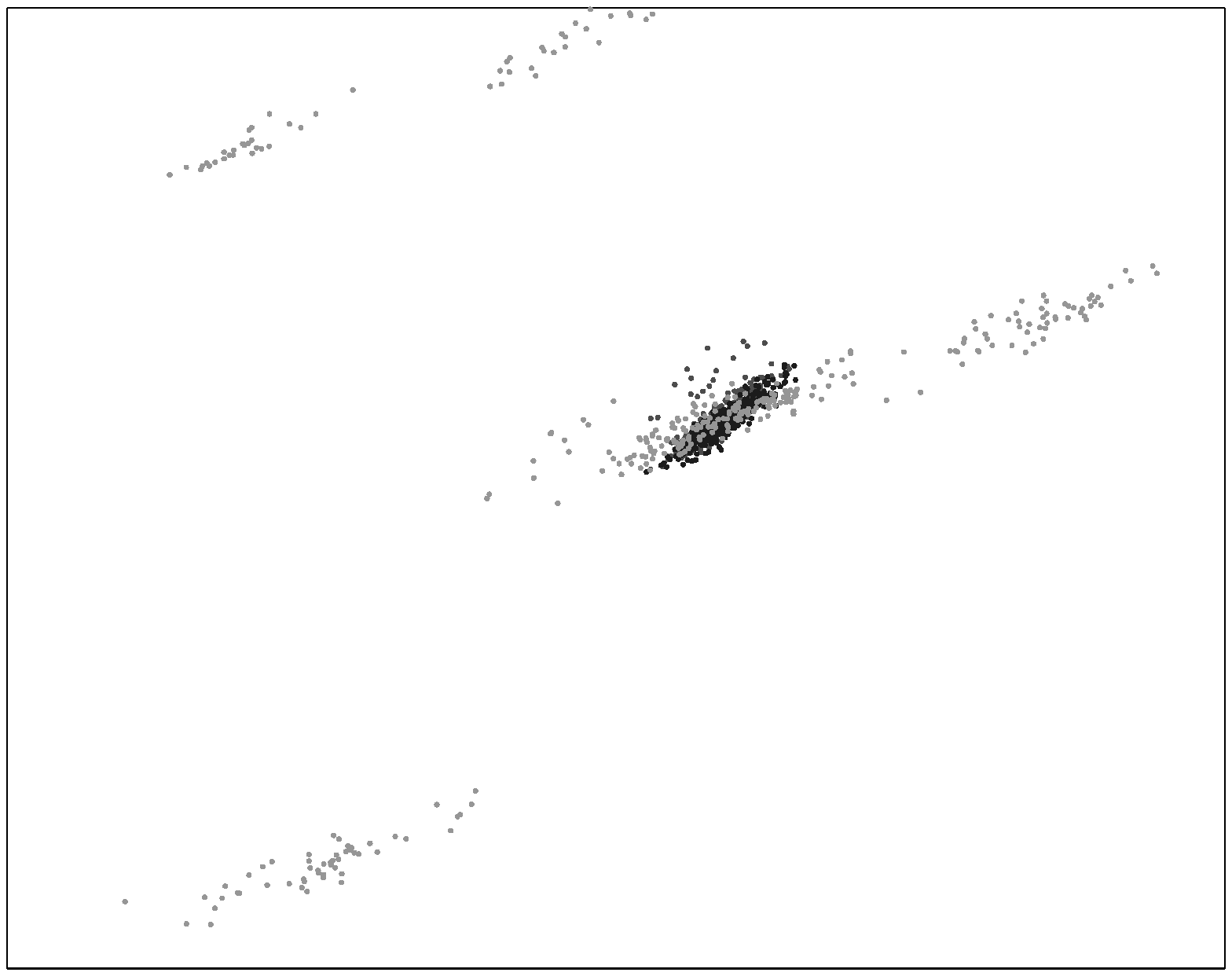}&
    \includegraphics[width=0.45\textwidth]{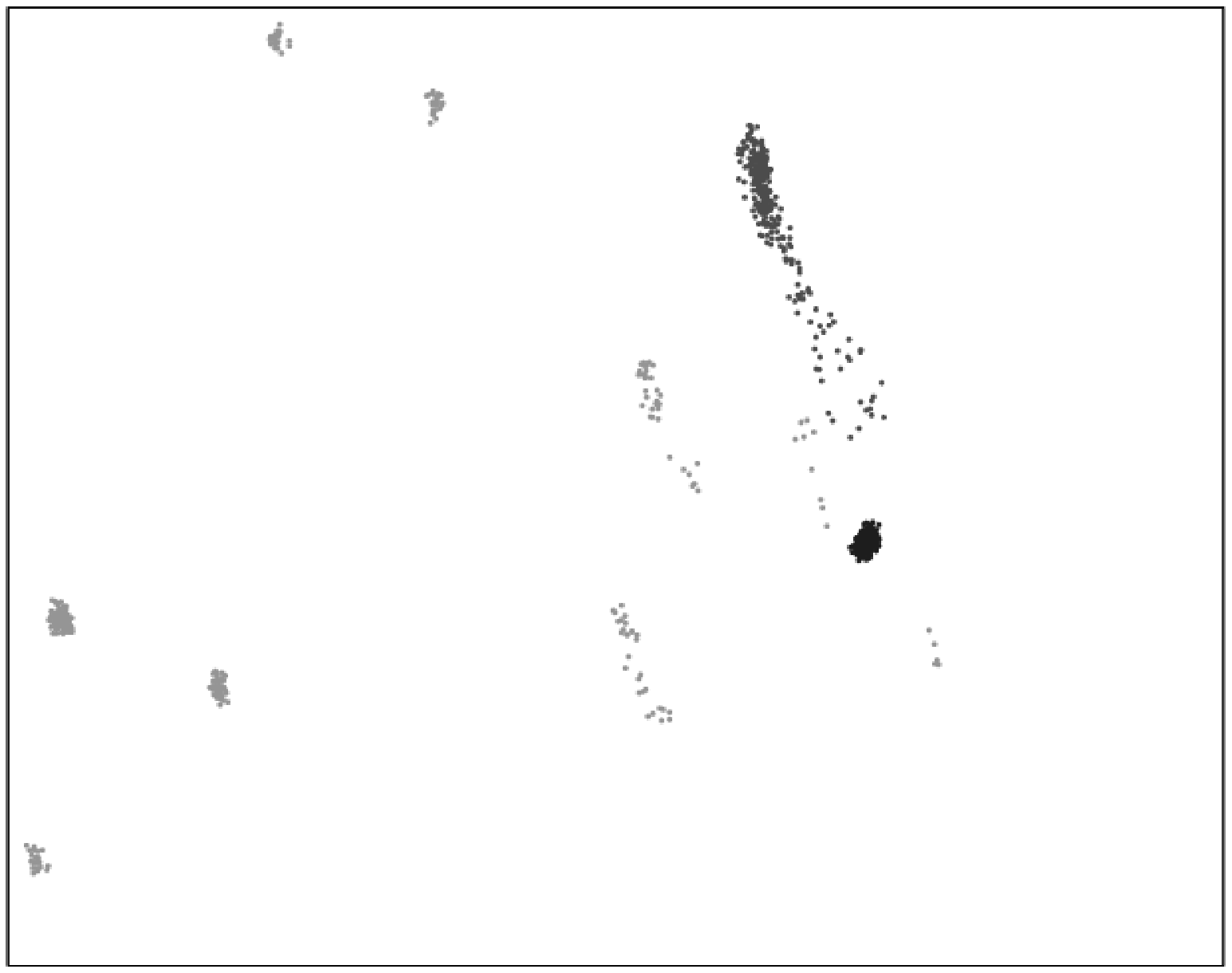}
   \end{tabular}
  \end{center}
  \caption{left: 2D projection of the ''oil flow'' data manually chosen from 3D projection obtained from by vanilla FastICA
  methods using the tanh index - right: projection obtained by NGCA using
  a combination of Fourier, tanh, Gauss-pow3 indices}
  \label{fig-oildataPP}
\end{figure}
Figure \ref{fig-oildataPP} shows a slice through \( \Gamma \) such that
the structure in the data set becomes visible: Using NGCA we can
distinguish \( 10-11 \) clusters versus at
most \( 5 \) for PP with index \rf{tanh}.\\

For the SNGCA method the results are shown in the figure
\ref{fig-oildataB}. SNGCA identifies 3 non-Gaussian dimensions. All
figures are rotated by hand such that the separation of the cluster is
illustrated at best. The next figure \ref{fig-oildataB} shows the result
of the oil-flow data obtained from SNGCA using a combination of the
indices \rf{tanh} and \rf{asymmeg}.
\begin{figure}[h]
    \centering
    \includegraphics[width=0.85\textwidth]{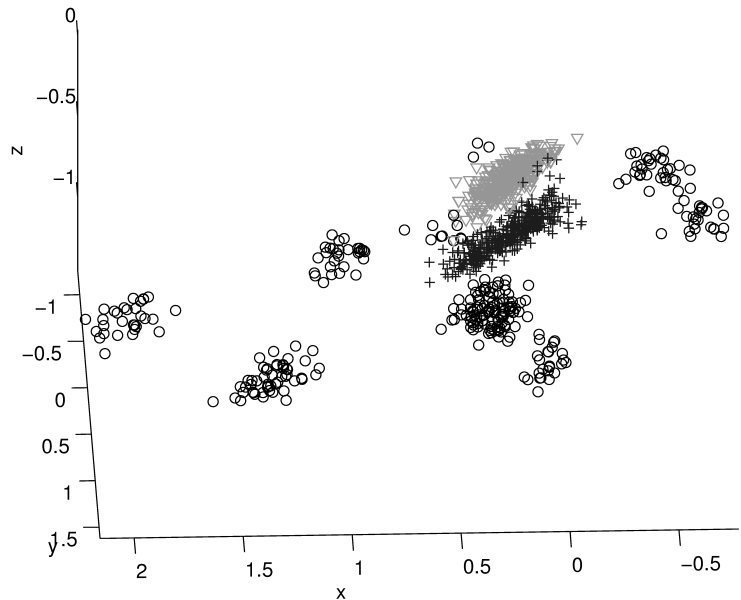}
    \caption{phase configurations of the ''oil flow'' data with apriori cluster mapping induced by crosses, circles and triangles obtained by SNGCA using a combination of asymmetric-Gauss and the tanh index\label{fig-oildataB}}
\end{figure}
In this case we can distinguish \( 10-11 \) clusters versus at most \( 5
\) for PP. Moreover we confirm the result of NGCA on the data set. The
clusters are clearly separated from each other on the SNGCA projection.
Only on the PP projection they are partially confounded in one single
cluster. By applying the projection obtained from SNGCA to the test data,
we found the cluster structure to be relevant. We conclude that SNGCA
gives a more relevant estimation of \( \cI \) than PP. However it is found
that the family of functions \( h_{\omega}(x) \) is an important tuning
parameter in SNGCA: If we use only the tanh-index, we found only 6-7
cluster are identified and they are partially confounded. Hence a
combination should be used in order to cope with symmetric data
distributions.

\section{Conclusion}

We propose a new improved methodology for the non-Gaussian component
analysis, as proposed in \cite{a3-3}. As well as NGCA the suggested method
is based on a semi-parametric framework for separation an uninteresting
multivariate Gaussian noise subspace from a linear subspace, where the
data are non-Gaussian distributed. Both methods assume that the
non-Gaussian contribution to the data density contains the structure in a
given data set. The combined strategy of convex projection and structural
adaptation provides promising results of SNGCA. Moreover SNGCA provides an
estimate for the dimension of the non-Gaussian subspace. On the other
hand, the quality and the numerical complexity of Monte-Carlo sampling of
the measurement directions is the main limitation of the proposed
technique.

\begin{appendix}
\section{Statistical tests}

In this section we shortly report the statistical tests on normality used
the dimension reduction step of SNGCA.

In order to detect a significant asymmetry in the distribution of the
original data projected on the semi-axis of the numerical approximation of
the rounding ellipsoid \(\cE_{\sqrt{d}}\) we use the \( K^{2} \)-test
according to D'Agostino-Pearson \cite{Zar99}. The D'Agostino-Pearson test
computes how far the empirical skewness and kurtosis of the given data
distribution differs from the value expected with a Gaussian distribution.
The test statistic is approximately distributed according to the \(
\chi^{2}_{2} \)-distribution and its empirical data counterpart is given
by
\begin{eqnarray*}
   \hat{K}^{2} &=& \cZ^{2}(\sqrt{b_{1}})+\cZ^{2}(b_{2})
   \\
   \sqrt{b_{1}} &=& \frac{1}{N}\sum_{i=1}^{N}[\sigma^{-1}(X_{i}-\mu)]^3
   \\
   b_{2} &=& \frac{1}{N}\sum_{i=1}^{N}[\sigma^{-1}(X_{i}-\mu)]^4
\end{eqnarray*}
Here \( \mu \) denotes the empirical mean, \( \sigma \) the empirical
standard deviation of the data and \( \cZ(\cdot) \) denotes a normalizing
transformations of skewness and kurtosis. The test is more powerful w.r.t.
an asymmetry of a distribution.\\

Furthermore we use the EDF-test according to Anderson-Darling
\cite{AnscombeGlynn83} with the modification of Stephens
\cite{Stephens86}: Let \( F_{N} \) be the empirical cumulative
distribution function and \( F \) the assumed theoretical cumulative
distribution function. The test statistics \( \cT \) measures the
quadratic deviations between \( F_{N} \) and \( F \):
\begin{eqnarray*}
    \cT=\int_{\R}[F_{N}(x)-F(x)]^{2}\nu(x)\;dF
\end{eqnarray*}
where \( \nu(x) \) is the weighting function \(
\nu(x)=[F_{N}(x)(1-F_{N}(x))]^{-1} \). In sum the data counterpart of \(
\cT \) is given by
\begin{eqnarray*}
 \hat{\cT}=&-cN-c\sum_{i=1}^{N}N^{-1}(2i-1)[\log(F(\sigma^{-1}(X_{i}-\mu))\\
 &+\log(1-F(\sigma^{-1}(X_{N-i+1}-\mu))]
\end{eqnarray*}
where \( c=1+0.75N^{-1}+2.25N^{-2} \). Again \( \mu \) is the empirical
mean and \( s \) the empirical standard deviation of the data. We compute
\( \hat{\cT} \) to detect deviations from normality in the tails of the
projected distributions. The test is rejected if \( \hat{\cT} \) exceeds a
critical value \( cv \) specific for a given level of significance:
\begin{center}
\centering
\begin{tabular}{|c|c|c|c|c|c|}\hline
  \( \alpha \) : &  \( 0.10 \)  & \( 0.05 \)  & \( 0.025 \) & \( 0.01 \)  & \( 0.005 \) \\\hline
  \( cv \) : &  \( 0.631 \) & \( 0.752 \) & \( 0.873 \) & \( 1.035 \) & \( 1.159 \) \\  \hline
\end{tabular}\\
\end{center}
The last test, applied to the projected data is the Shapiro-Wilks test
\cite{ShapiroWilks} based on a regression strategy in the version given by
Royston \cite{Royston82a,Royston82c}:
\begin{eqnarray*}
    W &=& \sigma^{-1}[1-b^{2}(\sigma^{2}(N-1))^{-1}]^{\lambda} \sim \cN(\mu,1)
    \\
    b &=& \sum_{i=1}^{N/2}a_{N-i+1}(X_{N-i+1}-x_{i})\\
    (a_{1},\dots,a_{N}) &=& {m^\top \Sigma^{-1} \over (m^\top \Sigma^{-1^\top}\Sigma^{-1}m)^{1/2}}
\end{eqnarray*}
In this test \( m=(m_{1},\ldots,m_{n}) \) denotes the expected values of
standard normal order statistics for a sample of size \( N \) and \(
\Sigma \) is the corresponding covariance matrix.

\section{Proofs}
\subsection{Proof  of Theorem \ref{empir1}}

We use the following result from the empirical process theory (similar
statements under slightly different assumptions can be found e.g. in
\cite{vdWaart96}). Let \( \cB \) stand for the unit Euclidean ball,
centered at the origin. Similarly,  \( B(\mu,\omegad) = \{ \omega : \|
\omega - \omegad \|_{2} \le \mu \} \) is a ball of radius \( \mu \)
centered at \( \omegad \). For a function \( q(\omega,x) \), denote \(
\E_{N} [q(\omega,X)] = N^{-1}\sum_{i=1}^{N} q(\omega,X_{i}) \).

\def\qox{q^{*}}
\def\qg{q^{*}}
\def\zz{\mathfrak{z}}

\begin{lemma}
\label{Lemp} Let \( q(\omega,x) \) be a continuously differentiable
function of \( \omega \in \cB_{d} \) and \( x \in \R^{d} \) such that for
every \( \omega \in \cB_{d} \)
\begin{eqnarray}
\label{nablaomega}
    \Var \bigl[ q(\omega,X) \bigr]
    \le
    \qox,
    \quad
    \Cov \bigl[ \nabla_{\omega} q(\omega,X) \bigr]
    \le
    \qg I,
\end{eqnarray}
with some \( \qox,\qg > 0 \). Define
\begin{eqnarray*}
     \zeta(\omega)
     =
     N^{1/2} \bigl\{ \E_{N} [q(\omega,X)] - \E [q(\omega,X)] \bigr\}
\end{eqnarray*}
and \( \zeta(\omega,\omegac) = \zeta(\omega) - \zeta(\omegac) \). Then for
any \( \kapla > 1 \), there is \( \lambdab_{1} = \lambdab_{1}(\kapla) > 0
\) such that for any \( \omegad \in \cB_{d} \), \( \mu \le 1 \), and \(
\lambda \le \lambdab_{1} N^{1/2} \)
\begin{eqnarray}
\label{firstexpb}
    \log \E \exp\bigl[ \lambda \zeta(\omegad) \bigr]
    \!\!\!
    & \le &
    \!\!\!
    \kapla \qox \lambda^{2} / 2,
    \\
    \log \E \exp\Bigl[
        \frac{\lambda}{\mu} \sup_{\omega \in B(\mu,\omegad)}
        \zeta(\omega,\omegad)
    \Bigr]
    \!\!\!
    & \le &
    \!\!\!
    2 \kapla \qg \lambda^{2} + \entrl_{d} \, ,
    \qquad
\label{secondexpb}
\end{eqnarray}
where \( \entrl_{d} = \sum_{k=1}^{\infty} 2^{-k} \log (2^{kd}) = 4 d \log
2 \). Moreover, define
\begin{eqnarray*}
    \zz(\lambda)
    =
    \kapla  \bigl( \qox/2 + 2 \qg \bigr) \lambda^{2} + \entrl_{d} .
\end{eqnarray*}
Then for any \( \varepsilon > 0 \)
\begin{eqnarray*}
    \P\biggl(
        \sup_{\omega \in \cB_{d}} \zeta(\omega)
        \ge
        2\lambda^{-1} \bigl[ \zz(\lambda) + \log \varepsilon^{-1} \bigr]
    \biggr)
    \le
    \varepsilon .
\end{eqnarray*}
\end{lemma}
\begin{proof}
{\noindent}Define for \( \omega \in \cB_{d} \)
\begin{eqnarray*}
    g_{0}(\lambda;\omega)
    =
    \log \E \exp \Big[
        \frac{\lambda}{\sqrt{\kapla \qox}} \bigl\{ q(\omega,X_{1})
        - \E[q(\omega,X_{1})]  \bigr\}
    \Big] .
\end{eqnarray*}
Then \( g_{0}(\lambda;\omega) \) is analytic in \( \lambda \) and
satisfies \( g_{0}(0;\omega) = g'_{0}(0;\omega) = 0 \). Moreover,  the
condition (\ref{nablaomega}) implies \( g''_{0}(0;\omega) < 1 \).
Therefore,  there is some \( \lambdab_{1} > 0 \) such that for any \(
\lambda_{1} \le \lambdab_{1} \) and any unit vector \( \omega \), it holds
\(
    g_{0}(\lambda_{1};\omega)
    \le
    \lambda_{1}^{2}/2
\). Independence of the \( X_{i} \)'s implies (\ref{firstexpb}) for \(
\lambda  \le \lambdab_{1} N^{1/2} (\kapla \qox)^{-1/2} \). In the same
way,  for \( \omega,u \in \cB_{d} \) define \( \zeta(\omega,X) =
\nabla_{\omega} q(\omega,X_{1})
        - \E[ \nabla_{\omega} q(\omega,X_{1})] \) and
\begin{eqnarray*}
    g(\lambda;\omega,u)
    = \log \E \exp \big[
        \frac{2 \lambda u^{\T} }{\sqrt{\kapla \qg}} \zeta(\omega,X_{1})
    \big] .
\end{eqnarray*}
Then similarly to the above, the function \( g(\lambda;\omega,u) \) is
analytic in \( \lambda \) and satisfies with some \( \lambdab_{1} > 0 \),
any \( \lambda_{1} \le \lambdab_{1} \) and any unit vectors \( u \) and \(
\omega \)
\begin{eqnarray*}
    g(\lambda_{1};\omega,u)
    \le
    2 \lambda_{1}^{2}.
\end{eqnarray*}
The bound (\ref{secondexpb}) is derived from \cite{Spexp2009}, Lemma~5.1.
Independence of the \( X_{i} \)'s yields for \( \lambda \le \lambdab_{1}
N^{1/2} (\kapla \qg)^{-1/2} \)
\begin{eqnarray*}
    \log \E \exp \biggl\{
        \frac{2 \lambda}{\sqrt{\kapla \qg}} u^{\T} \nabla \zeta(\omega)
    \biggr\}
    \le
    2 \lambda^{2} .
\end{eqnarray*}
This means that the condition \( (\cc{E}D) \) of \cite{Spexp2009} is
verified and the result (\ref{secondexpb}) follows from \cite{Spexp2009},
Lemma 5.1. Introduce a random set \( A = \{ (\lambda/2) \sup_{\omega}
\zeta(\omega) > \zz(\lambda)
        + \log \varepsilon^{-1} \} \).
and \( A^{c} \) is its complement. By the Cauchy-Schwartz inequality
\begin{eqnarray*}
    \P(A^{c})
    \!\!\!
    & \le &
    \!\!\!
    \E \exp\biggl\{
        \frac{\lambda}{2} \sup_{\omega} \zeta(\omega) - \zz(\lambda)
        - \log \varepsilon^{-1}
    \biggr\}
    \nn
    \!\!\! & \le & \!\!\!
    \varepsilon
    \E^{1/2} \exp\bigl\{ \lambda \zeta(\omegad) - \kapla \qox \lambda^{2}/2  \bigr\}
    \nn
    && \!\!\! \times \,
    \E^{1/2} \exp\bigl\{
        \lambda \sup_{\omega} \zeta(\omega,\omegad)
        - 2 \kapla \qg \lambda^{2} - \entrl_{d}
    \bigr\}
    \le \varepsilon
\end{eqnarray*}
and the last result follows.\\
\end{proof}
The result of Lemma~\ref{Lemp} can be easily extended to the case of a
vector function \( q(\omega,x) \in \R^{d} \):
\begin{eqnarray*}
    \P\biggl(
        \sup_{\omega \in \cB_{d}} \| \zeta(\omega) \|_{\infty}
        \ge
        2\lambda^{-1} \bigl[ \zz(\lambda) + \log (d /\varepsilon) \bigr]
    \biggr)
    \le
    \varepsilon .
\end{eqnarray*}
This fact can be obtained by applying Lemma~\ref{Lemp} to each component
of the vector \( \zeta(\omega) \). The term \( \log (d /\varepsilon) \) is
responsible for the overall deviation probability.\\

Let now \( f(x,\omega) \) be a twice continuously differentiable function
of \( \omega \in \cB_{d} \) and \( x \in \R^{d} \) such that for every \(
j \le d \), \( \omega \in \cB_{d} \), and \( x \in \R^{d} \), it holds
\begin{eqnarray*}
    \Var \bigl[ X_{j} \, f(X,\omega) \bigr]
    \le
    f^{*}_{1},
    \quad
    \Cov \bigl[ X_{j} \, \nabla_{\omega} f(X,\omega) \bigr]
    \le
    f^{*}_{1} I,
    \nn
    \Var \biggl[ \frac{\partial}{\partial x_{j}} f(X,\omega) \biggr]
    \le
    f^{*}_{1},
    \quad
    \Cov \biggl[ \nabla_{\omega} \frac{\partial}{\partial x_{j}} f(X,\omega) \biggr]
    \le
    f^{*}_{1} I,
\end{eqnarray*}
Then for any \( \kapla > 1 \), there is \( \lambdab_{1} =
\lambdab_{1}(\kapla) > 0 \) and for any \( \varepsilon > 0 \), a random
set \( A \) with \( \P(A) \ge 1 - \varepsilon \) such that on \( A \) it
holds by Lemma~\ref{Lemp}
\begin{eqnarray*}
    \sup_{\omega \in \cB_{d}}
    \big\| \E_{N} [X f(X,\omega)] - \E [X f(X,\omega)] \big\|_{\infty}
    \le
    \delta_{N},
    \nn
    \sup_{\omega \in \cB_{d}}
    \big\| \E_{N} [\nabla_{x} f(X,\omega)] - \E [\nabla_{x} f(X,\omega)] \big\|_{\infty}
    \le
    \delta_{N},
\end{eqnarray*}
where
\begin{eqnarray*}
    \delta_{N}
    =
    N^{-1/2} \inf_{\lambda \le \lambdab_{1} N^{1/2}}
    \bigl\{
        5 \kapla f^{*}_{1} \lambda
        + 2\lambda^{-1} \bigl[ \entrl_{d} + \log(2d/\varepsilon) \bigr]
    \bigr\}.
\end{eqnarray*}
By construction of vectors \( \hat{\gamma}_{l} \) and \( \hat{\eta}_{l}
\), it holds on \( A \)
\begin{eqnarray*}
    \max_{1 \le l \le L} \| \hat{\gamma}_{l} - \gamma_{l} \|_{\infty}
    \le
    \delta_{N},
    \quad
    \max_{1 \le l \le L} \| \hat{\eta}_{l} - \eta_{l} \|_{\infty}
    \le
    \delta_{N} \, .
\end{eqnarray*}
This implies for any \( \| c \|_{1} \le 1 \)
\begin{eqnarray*}
    \| \hat{\gamma}(c) - \gamma(c) \|_{\infty}
    \le
    \delta_{N},
    \;\;\;
    \| \hat{\eta}(c) - \eta(c) \|_{\infty}
    \le
    \delta_{N}.
\end{eqnarray*}
The constraint \( \hat{\gamma}(\hat{c}) = 0 \) implies \( \|
\gamma(\hat{c}) \|_{\infty} \le \delta_{N} \), thus
\[
    \| \gamma(\hat{c}) \|_{2}
    \le
    \sqrt{d} \, \delta_{N},
\]
and by \rf{bounds}
\begin{eqnarray*}
   \lefteqn{\bigl\| (I - \Pi^{*}) \hat{\eta}(\hat{c}) \bigr\|_{2}}
   \\
   & \le &
   \bigl\| (I - \Pi^{*}) \{ \hat{\eta}(\hat{c}) - {\eta}(\hat{c}) \} \bigr\|_{2}
   +
   \bigl\| (I - \Pi^{*}) \eta(\hat{c}) \bigr\|_{2}
   \\
   & \le &
   \bigl\| \hat{\eta}(\hat{c}) - \eta(\hat{c}) \bigr\|_{2}
   +
   \bigl\| \Sigma^{-1} \gamma(\hat{c}) \bigr\|_{2}
   \\
   & \le &
   \sqrt{d} \bigl( \delta_{N} + \bigl\| \Sigma^{-1} \bigr\|_{2} \delta_{N} \bigr).
\end{eqnarray*}
\subsection{Proof of Theorem \ref{round1}}
Let \( \cc{S} \) stand for the convex envelope of \(
\{\pm\hat{\beta}_{j}\}_{j=1}^J \). As \( \cc{E}_{1}(B) \) is inscribed in
\( \cc{S} \), its support function \( \xi_{\cc{E}_{1}(B)}(x)=\max_{s\in
\cc{E}_{1}(B)} s^{\T}x \) is majorated by that of \( \cc{S} \):
\[
    \xi_{\cc{E}_{1}(B)}(v)
    \le
    \xi_{\cc{S}}(v)
    =
    \max_{j=1,...,J} |v^{\T}\hat{\beta}_{j}|,\;\;
    \mbox{for any}\;\;v\in \R^{d}.
\]
Next, the support function of the ellipsoid \( \cc{E}_{1}(B) \) is
\[
    \xi_{\cc{E}_{1}(B)}(v)
    =
    (v^{\T}B^{-1}v)^{1/2},
\]
so that the condition \( \| \hat{\beta}_{j}-\beta_{j} \|_{2} \le \varrho
\) implies
\[
    v^{\T}B^{-1}v
    \le
    \max_{j=1,...,J} |v^{\T} \hat{\beta}_{j}|^{2}
    \le
    \bm^{2},
\]
for any \( v\perp \cI \).\\

Let us prove the second claim of the proposition. Let \( \Pi^{*} \) be a
projector onto \( \cI \). By the assumption of the proposition there exist
coefficients \( \mu_{j} \) with \( \sum_{j} \mu_{j} \le 1 \) such that
\begin{eqnarray*}
    S
    \eqdef
    \frac{1}{2} \biggl[
        \sum_{j} \mu_{j} \beta_{j} \beta_{j}^{\T} - 2\bm^{2}\Pi^{*}
    \biggr]
    \succeq 0.
\end{eqnarray*}
This implies (\ref{lambdam}). Now, for any such \( S \) and its
pseudo-inverse \( S^{+} \), the ellipsoid, \( \cc{E}^{f}_{1}(S^{+}) \)
with
\[
    \cc{E}^{f}_{1}(S^{+})
    =
    \{x\in \cI \mid x^{\T}S^{+}x\le 1\}
\]
is inscribed into \( \cc{S} \). Indeed, the support function \(
\xi_{\cc{E}^{f}_{1}(S^{+})}(x) = (x^{\T}Sx)^{1/2} \) of this ellipsoid
fulfills for \( x \in \cB_{d} \)
\begin{eqnarray*}
    \xi_{\cc{E}^{f}_{1}(S^{+})}(x)
    &\le&
    \biggl(
        \sum_{j} \mu_{j} \Bigl[ \frac{1}{2}(x^{\T}\beta_{j})^{2} - \bm^{2} \Bigr]
    \biggr)^{1/2}
    \\
    &\le&
    \biggl(
        \sum_{j} \mu_{j} \bigl| x^{\T}\hat{\beta}_{j} \bigr|^{2}
    \biggr)^{1/2}
    \\
    &\le&
    \max_{1\le j\le J} |x^{\T}\hat{\beta}_{j}|
    =
    \xi_\cc{S}(x),
\end{eqnarray*}
Now we are done: as the ellipsoid \( \cc{E}^f_{1}(S^{+}) \) is inscribed
into \( \cc{S} \), it is contained in the concentric to \( \cc{E}_{1}(B)
\) ellipsoid \( \cc{E}_{\sqrt{d}}(B) \) which covers \( \cc{S} \).\\

To show the last statement of the theorem, observe that
\begin{eqnarray*}
    \Tr\bigl[(\hat{\Pi}-\Pi^{*})^{2}\bigr]
    =
    2(m-\Tr[\Pi^{*}\hat{\Pi}])
    =
    2\Tr\bigl[(I-\Pi^{*})\hat{\Pi}\bigr].
\end{eqnarray*}
On the other hand, using the second claim one gets
\begin{eqnarray*}
    \Tr\bigl[(I-\Pi^{*})\hat{\Pi}\bigr]
    &\le&
    (d-m)\sup_{v\perp\cI} v^{\T}\hat{\Pi}v
    \\
    &\le &
    (d-m)\sup_{v\perp\cI} \frac{v^{\T}B^{-1}v}{\lambda_{m}(B^{-1})}
    \\
    &\le&
    \frac{2 d^{3/2}\bm^{2}}{ \lambda^{*}-2\bm^{2}}.
\end{eqnarray*}

\section{The algorithm}
Here we present the full algorithmic description of the SNGCA procedure.
We start with the linear estimation subprocedure:
\begin{algorithm2e}[H]\label{alg:estimation}
\SetLine \dontprintsemicolon \KwData{\( Y \),\( L \),\( J \)}
\KwResult{\( \{\hat{\beta}_{j}\}_{j=1}^J \)}%
{\bf Sampling}: choice of measurement directions\; %
\For{j=1 \KwTo J}{
    \For{l=1 \KwTo L}
    {
      Compute:\;
      \( \hat{\eta}_{jl}=N^{-1}\sum_{i=1}^N\nabla h_{\omega_{jl}}(Y_i) \)\;
      \( \hat{\gamma}_{jl}=N^{-1}\sum_{i=1}^N Y_i h_{\omega_{jl}}(Y_i) \)\;
    }
    Compute \( \hat{c}_{j} \) as in \rf{chat} and  \( \hat{\beta}_{j}=\sum_{l=1}^{L}\hat{c}_{j}\hat{\eta}_{jl} \).\;
} \caption{linear estimation of \( \beta(\psi_{h,c}) \)}
\end{algorithm2e}\\[2ex]
The following subprocedure reports the computation of the
$\sqrt{d}$-rounding ellipsoid based on a proposal in \cite{Nestreov04}:
\begin{algorithm2e}[H]\label{alg:ellipsoid}
\SetLine \dontprintsemicolon \KwData{\( \{\hat{\beta}_{j}\}_{j=1}^J \)}
\KwResult{\( \hat{B} \), }%
Let \( \delta_{i}^{k^{\ast}} = \max_{1\leq j\leq J} \;\langle
\hat{\beta}_{j},\hat{B}_{i} \hat{\beta}_{j} \rangle \) and set \(
\nu_{i}=\delta_{i}^{k^{\ast}}d^{-1} \).\; Let \( \hat{B}_0 \) be the
inverse empirical covariance matrix of the\; \( \hat{\beta}_{j} \) and set
\( t_{i}=\nu_{i}(\delta_{i}^{k^{\ast}}d^{-1}-1)^{-1} \).\\ Moreover let \( i \) be the loop index.\;%
\Repeat{\( \delta_i^{k^{\ast}}\leq C\cdot d \) where \( C \) is a tuning parameter.}%
{%
    \( x_{i}=\hat{B}_{i}\hat{\beta}_{k^{\ast}} \)\; %
    \( \hat{B}_{i+1}=(1-t_{i})^{-1}\Big(\hat{B}_{i}-t_{i}(1+\nu_{i})^{-1}x_{i}x_{i}^{\T} \Big) \)\; %
    \( \delta_{i+1}^{k^{\ast}}=(1-t_{i})^{-1}\Big(\delta_{i}^{k^{\ast}}-t_{i}(1+\nu_{i})^{-1}\langle \hat{\beta}_{k^{\ast}},x_{i}\rangle^{2} \Big) \)\; %
} \caption{Compute of the \( \sqrt{d} \)-rounding of the MVEE}
\end{algorithm2e}\\[2ex]

The next algorithm \ref{alg:dimRed} reports the pseudocode for
constructing a reduced basis of the target space from the estimated
elements by means of algorithm \ref{alg:ellipsoid}:
\begin{algorithm2e}[H]\label{alg:dimRed}
\SetLine \dontprintsemicolon \KwData{\( \hat{B} \)}
\KwResult{\( \big\langle \text{first } m \text{ eigenvectors of } \hat{B} \big\rangle \)}%
Let \( \hat{V} \) be the matrix of eigenvectors \( \hat{v}_{i} \) from \(
\hat{B} \)\;%
computed according to algorithm \ref{alg:ellipsoid}.\;
\For{i=1 \KwTo d}%
{%
   Project the data orthogonal on \( \hat{v}_{i} \).\;%
   Compute tests on normality of the projected data.\;%
} %
Discard every eigenvector with associated normal \;%
distributed projected data.\;%
\caption{Dimension Reduction}
\end{algorithm2e}\\[2ex]

In algorithm \ref{alg:estimation} we start with a random initialization of
the non-parametric estimator \( \hat{\beta}_{j} \) by means of a
Monte-Carlo sampling of the directions \( \omega_{jl} \) and \( \xi_{j}
\). However we can use the result of the first iteration \( k=1 \) of
SNGCA in order to accumulate information about \( \cI \) in a sequence \(
\hat{\cI}_{1},\hat{\cI}_{2},\ldots \) of estimators of the target space.
The procedure is described in detail in algorithm \ref{alg:adaptation}.\\[2ex]
\begin{algorithm2e}[H]\label{alg:adaptation}
\SetLine \dontprintsemicolon \KwData{\( \big\langle \text{first } m \text{ eigenvectors of } \hat{B} \big\rangle \)}\;%
Let \( \{ \hat{v_{i}} \}_{i=1}^m \) denote the reduced set of eigenvectors from \;%
\(\hat{B} \) and let \( k \) iterations be completed. To initialize
iteration \( k+1
\) choose random numbers \( z_{j1},\ldots,z_{jm} \) \;%
and \( u_{l1},\ldots,u_{lm} \) from \( \cU_{[-1,1]} \) and set \;%
\( \quad\quad\quad\xi_{j}:=\sum_{s=1}^{m} z_{js}\hat{v}_{i_s} \mbox{  for  } 1\leq j \leq n_{1}<J \)\;%
\( \quad\quad\quad\omega_{l}:=\sum_{s=1}^{m} u_{ls}\hat{v}_{i_s} \mbox{  for  } 1\leq l \leq n_{2} <L \)\;%
Then define \( \omega_{L-n_{2}},\ldots,\omega_{L} \) and \(
\xi_{J-n_{1}},\ldots,\xi_J \)
analogous to the case \( k=1 \). Now compose the sets\;%
\( \quad\quad\quad\{\xi_{1}^{(k)},\ldots,\xi_{n_{1}}^{(k)},\xi_{n_{1}+1}^{(k)},\ldots,\xi_J^{(k)}\} \) \;%
\( \quad\quad\quad\{\omega_{1}^{(k)},\ldots,\omega_{n_{2}}^{(k)},\omega_{n_{2}+1}^{(k)},\ldots,\omega_{L}^{(k)}\} \)\;%
For the initialization in the case \( k=k+1 \). Moreover we choose \(
n_{1}=kd \) and \( n_{2}=kd \) until \( n_{1} > J-d \) or \( n_{2} > L-d
\). Otherwise set \( n_{1} = J-d \) or
\( n_{2} = L-d \).\;%
\caption{structural adaptation of the linear estimation }
\end{algorithm2e}

{\bf Choice of parameters:} One of the advantages of the algorithm
proposed above is the fact that there are only a few tuning parameters.
\begin{itemize}
\item[i)] Suppose now that \( \omega_{i} \) is an absolute continuous random variable with
\( \omega_{i}\sim\cU_{[-1,1]} \). Without loss of generality we set \(
e=(1,0,\ldots,0) \).  Due to the normalization of \(
(\omega_{1},\ldots,\omega_{d}) \), it holds:
\begin{eqnarray*}
   \P\big( |(\omega_{1},\ldots,\omega_{d})^{\T}e | \geq 0.5\big)=\big(\sqrt{d}\big)^{-1}
\end{eqnarray*}
However the choice of \( J \) and \( L \) heavily depends on the
non-gaussian components. In the experiments we use \( 7d\leq J \leq 18d \)
and \( 6d\leq L \leq 16d \).
\item[ii)] Set the parameter of the stopping rule to \( \delta=0.05 \).
\item[iii)] Set the constant in the stopping rule for the computation of the MVEE to \( C=2 \).
\item[iv)] Set the significance level of the statistical tests to \( \alpha=0.05 \).\\
\end{itemize}

Finally we give a description of the complete algorithm.
\newpage

\begin{algorithm2e}[H]\label{alg:full}
\SetLine \dontprintsemicolon \KwData{\( \{X_i\}_{i=1}^N \),\( L \),\( J
\),\( \alpha \)}
\KwResult{\( \hat{\cI} \)}%
\BlankLine
{\bf Normalization:} The data \( (X_{i})_{i=1}^{N} \) are recentered. Let \;%
\(\sigma=(\sigma_{1},\ldots\sigma_{d}) \) be the standard deviations of the\;%
components of \( X_{i} \). Then \( Y_{i}=\diag(\sigma^{-1})X_{i} \) denotes the\;%
componentwise empirically normalized data.\;%
\BlankLine
{\bf Main Procedure:}\tcp*[r]{loop on \( k \)}%
\While{\( \sim StoppingCriterion(\cI,\hat{\cI}) \)}
{%
  \BlankLine
  {\bf Sampling:} The components of the {\em Monte-Carlo}-parts\;%
  of \( \xi_{j}^{(k)} \) and \( \omega_{jl}^{(k)} \) are randomly chosen from \( \cU_{[-1,1]} \).\;%
  The other part of the measurement directions are\;%
  initialized according to the structural adaptation\;%
  approach described in algorithm \ref{alg:adaptation}. Then \( \xi_{j}^{(k)} \) and\;%
   \( \omega_{jl}^{(k)} \) are normalized to unit length.\;%
   \BlankLine
   {\bf Linear Estimation Procedure:}\;%
    \For{j=1 \KwTo J}
    {\BlankLine
     \For{l=1 \KwTo L}
     {
       \( \hat{\eta}_{jl}^{(k)}=N^{-1}\sum_{i=1}^N\nabla h_{\omega_{jl}^{(k)}}(Y_i) \)\;
       \( \hat{\gamma}_{jl}^{(k)}=N^{-1}\sum_{i=1}^N Y_i h_{\omega_{jl}^{(k)}}(Y_i) \)\;
     }\BlankLine
     Compute the coefficients \( \{c_{l}\}_{l=1}^{L} \) by solving the\;%
     second-order conic optimization problem \rf{chat}:\;%
     \( \quad\quad\quad\quad\quad\min \; q  \qquad \mbox{s.t.} \)\;%
     \( \quad\quad\quad\quad\quad\quad\frac{1}{2}\|z\|_{2}\leq q \)\;%
     \( \quad\quad\sum_{l=1}^{L}(c_{l}^{+}-c_{l}^{-}) \hat\eta_{jl}^{(k)}-z=\xi_{j}^{(k)} \)\;%
     \( \quad\quad\quad\sum_{l=1}^{L} (c_{l}^{+}-c_{l}^{-}) \hat\gamma_{jl}^{(k)}=0 \)\;%
     \( \sum_{l=1}^{L} (c_{l}^{+}-c_{l}^{-})\leq 1, \quad 0\leq c_{l}^{+},c_{l}^{-} \quad \forall l \)\;%
     \BlankLine
     Compute \( \hat{\beta}_{j}^{(k)}=\sum_{l=1}^{L}(\hat{c}_{l}^{+}-\hat{c}_{l}^{-}) \hat{\eta}_{jl}^{(k)} \)\;
   \BlankLine
   }
   \BlankLine
   {\bf Dimension Reduction:}\;%
   Compute the symmetric matrix \( \hat{B}^{(k)} \) defining the\;%
   approximation of \( \cE \) according to algorithm \ref{alg:ellipsoid}.
   Reduce the basis of \( \cX \) according to algorithm \ref{alg:dimRed}.\;%
   \BlankLine
} \caption{full procedure of SNGCA}
\end{algorithm2e}\\[2ex]

{\bf Complexity:} We restrict ourselves to the leading polynomial terms of
the arithmetical complexity of corresponding computations counting only
the multiplications.
\begin{itemize}
 \item[1.] The numerical effort to compute \( \eta_{jl} \) and \( \gamma_{jl} \) in algorithm \ref{alg:estimation}
           heavily depends on the choice of \( h(\omega^{\T}x) \). Let \( h(\omega^{\T}x)=\tanh(\omega^{\T}x) \). Then
           this step takes \( \cO(J(\log N)^{2}N^{2}) \) operations.
 \item[2.] Algorithm \ref{alg:ellipsoid} takes \( \cO(d^{2}J\log(J)) \) operations \cite{Nestreov04}.
 \item[3.] For the optimization step in \ref{alg:estimation} we use a commercial solver\footnote{http://www.mosek.com}
           based on an interior point method. The constrained convex
           projection solved as an SOCP takes \( \cO(d^{2}n^3) \) operations there \( n \) is the
           number of constraints.
 \item[4.] Computation of the statistical tests in one dimension: Let \( N \) denote the
           number of samples. D'Agostino-Pearson-test needs \( \cO(N^3\log N)\) and the
           Anderson-Darling-test \( \cO((\log N)^{2}N^{2}) \) operations. The test of Shapiro-Wilks
           takes \( \cO(N^{2}) \). In order to avoid robustness problems \cite{Thode02}
           the number of samples is limited to \( N\leq 1000 \). For
           larger data sets, \( N=1000 \) points are randomly chosen.
\end{itemize}
Hence without tests \( \hat{\cI} \) is computed in \( \cO(J(\log
N)^{2}N^{2}+d^{2}J\log(J)+d^2n^3) \) arithmetical operations per iteration.

\section*{Acknowledgment}
We are grateful to Yuri Nesterov from the {\sc CORE},  Louvain-la-Neuve
for helpful discussions and Gilles Blanchard from the {\sc FIRST.IDA}
Fraunhofer Institute Berlin for the permission to republish the results of
NGCA.

\end{appendix}
\bibliography{literature}

\begin{thebibliography}{10}

\bibitem{AnscombeGlynn83}
F.J. Anscombe and W.J. Glynn.
\newblock Distribution of kurtosis statistic for normal statistics.
\newblock {\em Biometrika}, 70(1):227--234, 1983.

\bibitem{Cook2001}
E.~Bura and R.~D. Cook.
\newblock Estimating the structural dimension of regressions via parametric
  inverse regression.
\newblock {\em J. Roy. Statist. Soc. Ser. B}, 63(393-410), 2001.

\bibitem{GTM98}
M.~Svensen C.M.~Bishop and C.K.I. Wiliams.
\newblock Gtm: The generative topographic mapping.
\newblock {\em Neural Computation}, 10(1):215--234, 1998.

\bibitem{Cook1998}
R.D. Cook.
\newblock Principal hessian directions revisited.
\newblock {\em J. Am. Statist. Ass.}, 93:85--100, 1998.

\bibitem{Cover1991}
T.M. Cover and J.A. Thomas.
\newblock {\em Elements of Information Theory}.
\newblock Wiley Series in Telecommunications. Wiley and Sons, New York, 1991.

\bibitem{Friedman1984}
P.~Diaconis and D.~Friedman.
\newblock Asymptotics of graphical projection pursuit.
\newblock {\em Annals of Statistics}, 12(3):793--815, 1984.

\bibitem{Goljandina}
N.E. Goljandina, V.V. Nekrutkin, and A.A. Zhigljavsky.
\newblock {\em Analysis of Time Series Structure: SSA and related technique}.
\newblock Chapman and Hall (CRS), Boca Raton, 2001.

\bibitem{Hastie01}
T.~Hastie, R.~Tibshirani, and J.~Friedman.
\newblock {\em The elements of statistical learning}.
\newblock Springer Series in Statistcs. Springer, 2001.

\bibitem{a3-1}
M.~Hristache, A.~Juditsky, J.~Polzehl, and V.~Spokoiny.
\newblock Structure adaptive approach for dimension reduction.
\newblock {\em Ann. Statist.}, 29(6):1537--1566, 2001.

\bibitem{Huber85}
P.~J. Huber.
\newblock Projection pursuit.
\newblock {\em The Annals of Statistics}, 13(2):435--475, 1985.

\bibitem{Hyvaer99}
A.~Hyv\"arinen.
\newblock Survey on independent component analysis.
\newblock {\em Neural Computing Surveys}, 2:94--128, 1999.

\bibitem{John48}
F.~John.
\newblock {\em Extremum problems with inequalities as subsidiary conditions},
  volume Reprinted in: Fritz John, Collected Papers Volume 2 of {\em
  Birkhäuser, Boston}, pages 543--560.
\newblock J. Moser, 1985.

\bibitem{Jolliffe01}
I.T. Jolliffe.
\newblock {\em Principal Component Analysis}.
\newblock Springer Series in Statistics. Springer, Berlin and New York, 2nd
  edition, 2002.

\bibitem{Thode02}
H.C.~Thode Jr.
\newblock {\em Testing for Normality}.
\newblock Marcel Dekker, New York., 2002.

\bibitem{Li1991}
K.C. Li.
\newblock Sliced inverse regression for dimension reduction.
\newblock {\em J. Am. Statist. Ass.}, 86:316--342, 1991.

\bibitem{Li1992}
K.C. Li.
\newblock On principal hessian directions for data visualisation and dimension
  reduction: another application of stein's lemma.
\newblock {\em Ann. Statist.}, 87:1025--1039, 1992.

\bibitem{Mizuta04}
M.~Mizuta.
\newblock {\em Dimension Reduction Methods}, chapter~6, pages 566--89.
\newblock J.E. Gentle and W. H\"ardle, and Y. Mori (eds.): Handbook of
  Computational Statistics, 2004.

\bibitem{Nestreov04}
Yu.~E. Nesterov.
\newblock Rounding of convex sets and efficient gradient methods for linear
  programming problems.
\newblock {\em Discussion Paper 2004-4, CORE, Catholic University of Louvain,
  Louvain-la-Neuve, Belgium}, 2004.

\bibitem{Roweis2000}
S.~Roweis and L.~Saul.
\newblock Nonlinear dimensionality reduction by locally linear embedding.
\newblock {\em Science}, 290:2323--2326, 2000.

\bibitem{Royston82a}
J.P. Royston.
\newblock An extension of shapiro and wilks' w test for normality to large
  samples.
\newblock {\em Applied Statistics}, 31:115--124, 1982.

\bibitem{Royston82c}
J.P. Royston.
\newblock The w test for normality.
\newblock {\em Applied Statistics}, 21:176--180, 1982.

\bibitem{ShapiroWilks}
S.S. Shapiro and M.B. Wilk.
\newblock An analysis of variance test for normality.
\newblock {\em Biometrika}, 52:591--611, 1965.

\bibitem{Spexp2009}
V.~Spokoiny.
\newblock A penalized exponential risk bound in parametric estimation.
\newblock {\em http://arxiv.org/abs/0903.1721}, 2009.

\bibitem{a3-3}
V.~Spokoiny, G.~Blanchard, M.~Sugiyama, M.Kawanabe, and Klaus-Robert M\"uller.
\newblock In search of non-{G}aussian components of a high-dimensional
  distribution.
\newblock {\em Journal of Machine Learning Research}, preprint TR05-003, 2005.

\bibitem{Stephens86}
M.~A. Stephens.
\newblock {\em Goodness of Fit Techniques}, chapter Tests based on Goodness of
  Fit.
\newblock D'Agostino, R. B. and Stephens, M. A., 1986.

\bibitem{vdWaart96}
A.~van~der Vaart and J.A. Wellner.
\newblock {\em Weak Convergence and Empirical Proccesses}.
\newblock Springer Series in Statistics. Springer -- New York, 1996.

\bibitem{Wasserman}
L.~Wasserman.
\newblock {\em All of Nonparametric Statistics}.
\newblock Springer Texts in Statistcs. Springer, 2006.

\bibitem{Wold82}
H.~Wold.
\newblock {\em Soft Modeling. The Basic Design and Some Extensions.}, volume~2
  of {\em Systems Under Indirect Observation}, pages 1--53.
\newblock K.-G. Jöreskog and H. Wold, North-Holland, Amsterdam, 1982.

\bibitem{Wold87}
S.~Wold, S.~Hellberg, M.~Sjostrom, and H.~Wold.
\newblock {\em PLS Model Building: Theory and applications. PLS modeling with
  latent variables in two or more dimensions.}
\newblock 1987.

\bibitem{a3-5}
Y.~Xia, H.~Tong, W.K. Li, and Li-Xing Zhu.
\newblock An adaptive estimation of dimension reduction space.
\newblock {\em Journal of the Royal Statistical Society, Series B},
  64(3):363--388, 2001.

\bibitem{Zar99}
J.H. Zar.
\newblock {\em Biostatistical Analysis, (2nd ed.)}.
\newblock NJ: Prentice-Hall, Englewood Cliffs., 1999.

\end{thebibliography}
\bibliographystyle{plain}

\end{document}